\documentclass[A4paper,10pt]{article}
\usepackage{amsmath,amssymb,amsfonts,amsthm}
\usepackage{enumerate}
\usepackage{paralist}
\usepackage{graphics} 
\usepackage{epsfig} 
\usepackage{graphicx}  \usepackage{epstopdf}
\usepackage[colorlinks=true]{hyperref}
\hypersetup{urlcolor=blue, citecolor=red}

\usepackage{bm}
\usepackage{color}

\newtheorem{theorem}{Theorem}[section]
\newtheorem{corollary}{Corollary}

\newtheorem{lemma}[theorem]{Lemma}
\newtheorem{proposition}{Proposition}

\theoremstyle{definition}
\newtheorem{definition}[theorem]{Definition}
\newtheorem{remark}{Remark}

\newcommand{\vnorm}[1]{\left\| #1 \right\|}
\newcommand{\scalarp}[1]{\left\langle #1 \right\rangle}

\title{(Un)conditional consensus emergence under feedback controls}

\date{}

\begin{document}
\maketitle

\vspace{-1cm}
\centerline{\scshape Mattia Bongini and Massimo Fornasier }
\medskip
{\footnotesize
 \centerline{Technische Universit\"at M\"unchen, Fakult\"at Mathematik}
   \centerline{ Boltzmannstra\ss e 3, D-85748 Garching, Germany}
} 

\medskip

\centerline{\scshape Dante Kalise }
\medskip
{\footnotesize
 \centerline{ Radon Institute for Computational and Applied Mathematics, Austrian Academy of Sciences }
   \centerline{Altenbergerstra\ss e 69, A-4040 Linz, Austria}
}

\medskip

\begin{abstract}
We study the problem of consensus emergence in multi-agent systems via external feedback controllers. We consider a set of agents interacting with dynamics given by  a Cucker-Smale type of model, and study its consensus stabilization by means of centralized and decentralized control configurations. We present a characterization of consensus emergence for systems with different feedback structures, such as leader-based configurations, perturbed information feedback, and  feedback computed upon spatially confined information. We characterize consensus emergence for  this latter design as a parameter-dependent transition regime between self-regulation and centralized feedback stabilization. Numerical experiments illustrate the different features of the proposed designs.
\end{abstract}

\section{Introduction}\label{sec:first_results}

Over the last years, the study of multi-agent systems has become a topic of increasing interest in mathematics, biology, sociology, and engineering, among many other disciplines. Upon the seminal articles of Reynolds \cite{Reynolds87}, Vicsek  et al. \cite{vics} and more recently, Cucker and Smale \cite{CS}, there has been a substantial amount of mathematical works addressing from both analytical and computational perspectives, different phenomena arising in this class of systems. Multi-agent systems are usually modeled as a large-set of particles interacting under simple binary rules, such as attraction, repulsion, and alignment forces, which can depend either metrically or topologically  on the agent configuration; the wide applicability of this setting ranges from modeling the collective behavior of bird flocks \cite{Reynolds87}, to the study of data transmission over communication networks \cite{ignaciuk2013congestion}, including the description of opinion dynamics in human societies \cite{krause02}, and the formation control of platoon systems \cite{murray07,Peters}.
At a microscopic level, multi-agent systems are often represented by a large-scale set of differential (or difference) equations; in this context, it is of interest the study of asymptotic behaviors, pattern formation, self-organization phenomena, and its basins of attraction.
To make matters concrete, in this article we consider a set of $N$, $d-$dimensional agents interacting under a controlled Cucker-Smale model of the form
\begin{eqnarray}
\left\{
\begin{aligned}
\begin{split} \label{eq:cuckersmale}
\dot{\bm{x}}_{i} & = \bm{v}_{i} \\
\dot{\bm{v}}_{i} & = \frac{1}{N} \sum_{j = 1}^N a\left(r_{ij}\right)\left(\bm{v}_{j}-\bm{v}_{i}\right) +\bm{u}_i, \qquad i = 1, \ldots, N
\end{split}
\end{aligned}
\right.
\end{eqnarray}
where the pair  $(\bm{x}_i,\bm{v}_i)\in\mathbb{R}^{2d}$ represents the position and velocity of every agent, $\bm{u}_i$ is an external controller to be suitably defined later, and  $a:[0, +\infty) \rightarrow [0, +\infty)$ is a bounded, non increasing, continuous function, whereas $r_{ij}$ stands for the Euclidean distance $\vnorm{x_i - x_j}$. This kind of model has been introduced by Cucker and Smale in \cite{CS,cusm07}, for specific choices of the interaction function $a$ as
\begin{equation}\label{eq:kernel}
a(r_{ij})=\frac{1}{(1+r_{ij}^2)^{\delta}} \quad \text{for every} \quad \delta \in [0,+\infty),
\end{equation}
later generalized to arbitrary positive interaction functions in \cite{HaHaKim}. Throughout our work we stick to this more general approach; we stress that henceforth, every  result we obtain applies for \emph{any} choice of the interaction kernel $a$ that is a positive, continuous, bounded and non increasing function.

For a group of agents evolving according to  \eqref{eq:cuckersmale} , we shall be concerned with studying the asymptotic convergence of the velocity field of to a common vector, a phenomenon often defined as consensus. It is clear that if the group converges to consensus, the consensus velocity coincides with the mean velocity of the group
\begin{equation}
\overline{\bm{v}} \stackrel{\Delta}{=}  \frac{1}{N} \sum^N_{i = 1} \bm{v}_i.
\end{equation}

\begin{definition}[Consensus] \label{def:consensus}
We say that a solution $(\bm{x}(t),\bm{v}(t))$ of system \eqref{eq:cuckersmale} \emph{tends to consensus} if the consensus parameter vectors $\bm{v}_i$ tend to the mean $\overline{\bm{v}}$, i.e.,
\begin{align*}
\lim_{t \rightarrow + \infty} \vnorm{\bm{v}_i(t) - \overline{\bm{v}}(t)} = 0 \quad \text{for every } i = 1,\ldots,N.
\end{align*}
\end{definition}

To set our work in perspective, let us begin by referring to the available results concerning consensus emergence for the uncontrolled system  \eqref{eq:cuckersmale},  i.e. when $\bm{u}_i\equiv 0$. In \cite{cusm07} a first result was presented related to the parameter $\delta$ in \eqref{eq:kernel}; it asserts that for $\delta\leq 1/2$, the system will tend asymptotically to consensus, independently of its initial configuration. For $\delta>1/2$, consensus emergence will depend on the cohesion of the initial setting.  A precise characterization of this situation was obtained in \cite[Theorem 3.1]{HaHaKim}, where the authors give a sufficient condition depending on the initial configuration and the parameter $\delta$. Further results concerning variations of the original system and consensus emergence have been presented in \cite{CuGu}, where the authors study the effect of adding agents with preferred navigation directions (stubborn agents), and \cite{hahakim10rayleigh}, where a Rayleigh-type of damping is added to the dynamics. 

 In general, the aforementioned results can be interpreted, in a wider framework, as stability results for nonlinear systems around a consensus manifold. A natural extension is then to consider the case when a controller is included as in \eqref{eq:cuckersmale}, and consensus can be achieved not only by internal self-regulation, but also by means of an external action. In \cite{CFPT}, the authors consider consensus stabilization for the Cucker-Smale system by means of both feedback-based controllers and open-loop, sparse optimal control. In particular, in \cite[Proposition 2]{CFPT}  it is shown that, with a controller of the form
\begin{equation}\label{eq:feed}
\bm{u}_i=-(\bm{v}_i-\bar{\bm{v}})\,,
\end{equation}
consensus emergence can be guaranteed for any configuration and values of $\delta$. A natural drawback of such a controller relates to the fact that it is always active, requiring what we call \textsl{full information or centralized control}, i.e., for a single agent, feedback computation will make use, at every time, of the total velocity field; even if full information is available, it is also possible that some perturbation is present. A much more realistic setting relates to what is known in the literature as \textsl{decentralized control} \cite{Bakule}; in our context, it means a control design where every agent acts based on partial information as for instance, the agents around a certain metrical or topological neighborhood \cite{Ballerini,olfati}.

The aim of this article is to make a contribution along these directions. Starting from a control of the form \eqref{eq:feed}, we will study variations of the feedback structure for consensus stabilization, by covering different settings such as feedback under perturbed information, leader-following feedback, and decentralized, local feedback depending on a metrical neighborhood of the agents. In every case, we present results concerning sufficient conditions for consensus emergence. In general, these results represent a transition between consensus emergence conditions for the uncontrolled Cucker-Smale system, and the feedback stabilization result under full information presented in \cite[Proposition 2]{CFPT}.

The paper is structured as follows. In Section 2,  we present some preliminary definitions and results concerning consensus emergence in the Cucker-Smale model. In Section 3 and 4, we introduce consensus stabilization results with feedback controllers based on perturbed information. Section 5 addresses the problem of consensus emergence under local feedback. Finally, in Section 6, numerical experiments are presented, aimed at illustrating the main features of the proposed designs; in particular, we numerically investigate the sharpness of the already existing and new estimates for local feedback stabilization.

\section{Preliminaries}
In order to start studying system \eqref{eq:cuckersmale},  we introduce the following notation and terminology: given a vector $\bm{a} = (\bm{a}_1, \ldots, \bm{a}_N) \in (\mathbb{R}^d)^N$, the symbol
\begin{align*}
\bm{a}^{\perp}_i \stackrel{\Delta}{=}  \bm{a}_i - \overline{\bm{a}}
\end{align*}
shall stand for the deviation of the vector $\bm{a}_i$ with respect to the mean $\overline{\bm{a}}$. Note that
\begin{equation*}
\sum^N_{i = 1} \bm{a}^{\perp}_i = \sum^N_{i = 1} \bm{a}_i - N \overline{\bm{a}} = N \overline{\bm{a}} - N \overline{\bm{a}} = 0,
\end{equation*}
and thus, for any vector $\bm{c} \in \mathbb{R}^d$, denoting by $\scalarp{\cdot,\cdot}$ the usual scalar product on $\mathbb{R}^d$, it holds
\begin{align} \label{eq:vertequalzero}
\sum^N_{i = 1} \scalarp{\bm{a}^{\perp}_i, \bm{c}} = \scalarp{\sum^N_{i = 1}\bm{a}^{\perp}_i, \bm{c}} = 0.
\end{align}

The following calculation shall be often exploited: given a $N\times N$ matrix $\omega$ which is symmetric and with positive entries, i.e., $\omega_{ij} = \omega_{ji}$ and $\omega_{ij} > 0$, for any $\bm{a} \in (\mathbb{R}^d)^N$ we have
\begin{equation} \label{eq:maintrick}
\begin{split}
\frac{1}{N^2} \sum^N_{i = 1} \sum^N_{j = 1} \omega_{ij} \scalarp{\bm{a}_j - \bm{a}_i, \bm{a}_i} & = \frac{1}{2 N^2} \Bigg( \sum^N_{i = 1} \sum^N_{j = 1} \omega_{ij} \scalarp{\bm{a}_j - \bm{a}_i, \bm{a}_i} \\
& \quad + \sum^N_{j = 1} \sum^N_{i = 1} \omega_{ji} \scalarp{\bm{a}_i - \bm{a}_j, \bm{a}_j} \Bigg) \\
& = -\frac{1}{2 N^2}\sum^N_{j = 1} \sum^N_{i = 1} \omega_{ij} \vnorm{\bm{a}_i - \bm{a}_j}^2 \\
& \leq - \min_{i,j} \omega_{ij} \frac{1}{N} \sum^N_{i = 1}\vnorm{\bm{a}_i^{\perp}}^2. \\
\end{split}
\end{equation}
In order to characterize consensus emergence in terms of the solutions $(\bm{x}(t),\bm{v}(t))$ of the system \eqref{eq:cuckersmale}, we define the following quantities
\[
X(t)=\frac{1}{2N^2}\sum_{i,j,=1}^N||\bm{x}_i(t)-\bm{x}_j(t)||^2\,,\quad\text{and}\quad V(t)=\frac{1}{2N^2}\sum_{i,j,=1}^N||\bm{v}_i(t)-\bm{v}_j(t)||^2\,,
\]
which provide an accurate description of consensus in terms of energy of the system by measuring the spread, both in positions and velocities, of the configuration.
A first result establishing a link between consensus in the sense of Definition \ref{def:consensus} and the above introduced quantities is stated as follows.

\begin{proposition}
The following are equivalent:
\begin{enumerate}
\item $\lim_{t \rightarrow + \infty} \vnorm{\bm{v}_i(t) - \overline{\bm{v}}(t)} = 0$ for every $i = 1,\ldots,N$,
\item $\lim_{t \rightarrow + \infty} \bm{v}_i^{\perp} = 0$ for every $i = 1,\ldots,N$,
\item $\lim_{t \rightarrow + \infty} V(t) = 0$.
\end{enumerate}
\end{proposition}

It is thus natural to prove a sufficiently strong decay of the functional $V(t)$ in order to establish that a solution of \eqref{eq:cuckersmale} tends to consensus. On the other hand, it is well known that not every solution of system \eqref{eq:cuckersmale} tends to consensus in the sense of Definition \ref{def:consensus}; in this context, a relevant result will be the characterization of consensus emergence introduced in \cite[Theorem 3.1]{HaHaKim}, which we recall in a concise version:
\begin{theorem}\label{thm:hhk}
Let $(\bm{x}_0, \bm{v}_0) \in (\mathbb{R}^d)^N \times (\mathbb{R}^d)^N$ and set $X_0 = B(\bm{x}_0, \bm{x}_0)$ and $V_0 = B(\bm{v}_0, \bm{v}_0)$. If the following inequality is satisfied:
\begin{align}
\label{eq:HaHaKim}
\int^{+\infty}_{\sqrt{X_0}} a(\sqrt{2N} r) \ dr \geq \sqrt{V_0},
\end{align}
then the solution of \eqref{eq:cuckersmale} with initial datum $(\bm{x}_0, \bm{v}_0)$ tends to consensus.
\end{theorem}

In general, we can induce consensus into the system, by adding a feedback term measuring the distance from the group velocity, leading to

\begin{eqnarray}
\left\{
\begin{aligned}
\begin{split} \label{eq:cuckersmale_uniform}
\dot{\bm{x}}_{i} & = \bm{v}_{i} \\
\dot{\bm{v}}_{i} & = \frac{1}{N} \sum_{j = 1}^N a\left(r_{ij}\right)\left(\bm{v}_{j}-\bm{v}_{i}\right) + \gamma(\overline{\bm{v}} - \bm{v}_i),
\end{split}
\end{aligned}
\right.
\end{eqnarray}
where $\gamma$ is a prescribed nonnegative constant, modeling the strength of the additional alignment term. As \eqref{eq:cuckersmale_uniform} can be rewritten as \eqref{eq:cuckersmale} with the function $a(r_{ij}) + \gamma$ replacing $a(r_{ij})$, by Theorem \ref{thm:hhk} each solution of \eqref{eq:cuckersmale_uniform} tends to consensus. As pointed out in \cite{CFPT}, the main drawback of this approach, however, is that it requires that each agent has a perfect information at every instant of the whole system, a condition which is seldom met in real-life situations; it is perhaps more realistic to ask that each agent computes an approximated mean velocity vector $\overline{\bm{v}}_i$, instead of the true mean velocity of the group $\overline{\bm{v}}$. Therefore, we consider the model

\begin{eqnarray}
\left\{
\begin{aligned}
\begin{split} \label{eq:cuckersmale_local}
\dot{\bm{x}}_{i} & = \bm{v}_{i} \\
\dot{\bm{v}}_{i} & = \frac{1}{N} \sum_{j = 1}^N a\left(r_{ij}\right)\left(\bm{v}_{j}-\bm{v}_{i}\right) + \gamma(\overline{\bm{v}}_i - \bm{v}_i).
\end{split}
\end{aligned}
\right.
\end{eqnarray}

In studying under which conditions the solutions of system \eqref{eq:cuckersmale_local} tend to consensus, it is often desirable to express the approximated feedback as a combination of a term consisting on a \textsl{true information feedback}, i.e., a feedback based on the real average $\overline{\bm{v}}$, and a perturbation term. We rewrite the system \eqref{eq:cuckersmale_local} in the following form:
\begin{eqnarray}
\left\{
\begin{aligned}
\begin{split} \label{eq:cuckersmale_perturbed}
\dot{\bm{x}}_{i} & = \bm{v}_{i} \\
\dot{\bm{v}}_{i} & = \frac{1}{N} \sum_{j = 1}^N a\left(r_{ij}\right)\left(\bm{v}_{j}-\bm{v}_{i}\right) + \alpha(\overline{\bm{v}} - \bm{v}_i) + \beta \Delta_i,
\end{split}
\end{aligned}
\right.
\end{eqnarray}
where $\alpha = \alpha(t)$ and $\beta = \beta(t)$ are two nonnegative, piecewise continuous functions, and $\Delta_i$ is a time-dependent not necessarily continuous deviation acting on agent $i$; therefore, solutions in this context have to be understood in terms of weak solutions in the Carath\'eodory sense \cite{filipov} (we also refer the reader  to  \cite[Appendix]{FS13} for specific details in the context of multi-agent systems). System \eqref{eq:cuckersmale_perturbed} encompasses all the previously introduced models, as it is readily seen:
\begin{itemize}
\item if $\alpha = \beta \equiv \gamma$ and $\Delta_i =  \bm{v}_i - \overline{\bm{v}}$, or $\alpha = \beta = 0$, then we recover system \eqref{eq:cuckersmale},
\item the choices $\alpha = \gamma$, $\Delta_i = 0$ (or equivalently $\beta = 0$) yield system \eqref{eq:cuckersmale_uniform},
\item if $\alpha = \beta \equiv \gamma$ and $\Delta_i = \overline{\bm{v}}_i - \overline{\bm{v}}$ we obtain system \eqref{eq:cuckersmale_local}.
\end{itemize}

The introduction of the perturbation term in system \eqref{eq:cuckersmale_perturbed} may deeply modify  the nature of the original model: for instance, an immediate consequence is that the mean velocity of the system is, in general, no longer a conserved quantity (as it is for system \eqref{eq:cuckersmale}).

\begin{proposition} \label{prop:derivative_mean}
For system \eqref{eq:cuckersmale_perturbed}, with perturbations given by the vector $\Delta = (\Delta_1, \ldots, \Delta_N)$, we have
\begin{align*}
	\frac{d}{dt}\overline{\bm{v}} = \beta \overline{\Delta}.
\end{align*}
\end{proposition}
\begin{proof}
\begin{equation*}
\begin{split}
	\frac{d}{dt} \overline{\bm{v}} & =  \frac{1}{N} \sum^N_{i = 1} \frac{d}{dt} \bm{v}_i \\
	& =  \frac{1}{N} \sum^N_{i = 1} \left(\frac{1}{N} \sum_{j = 1}^N a\left(r_{ij}\right)\left(\bm{v}_{j}-\bm{v}_{i}\right) + \alpha(\overline{\bm{v}} - \bm{v}_i) + \beta \Delta_i \right) \\
	& = \underbrace{\frac{1}{N^2} \sum^N_{i = 1} \sum_{j = 1}^N a\left(r_{ij}\right)\left(\bm{v}_{j}-\bm{v}_{i}\right)}_{= 0 \text{, by simmetry.}} + \underbrace{\frac{\alpha}{N} \sum^N_{i = 1} \bm{v}^{\perp}_i}_{= 0} + \frac{\beta}{N} \sum^N_{i = 1} \Delta_i \\
	& = \beta \overline{\Delta}.
\end{split}
\end{equation*}
\end{proof}

\begin{remark} \label{rem:derivative_mean}
As we have already pointed out, it is possible to recover system \eqref{eq:cuckersmale} by setting$\Delta_i =  \bm{v}_i - \overline{\bm{v}}$ whereas we can recover system \eqref{eq:cuckersmale_uniform} for the choice $\Delta_i =  0$. Note that in both cases we have $\overline{\Delta} =  0$, and therefore the mean velocity is a conserved quantity both in systems \eqref{eq:cuckersmale} and \eqref{eq:cuckersmale_uniform}.

We also highlight the fact that $\overline{\bm{v}}$ is not conserved even in the case that for every $t \geq 0,$ and for every $1 \leq i \leq N$, $\Delta_i(t) = \bm{c}$, where $\bm{c} \not = 0$, i.e., the case in which each agent makes the same error in evaluating the mean velocity.
\end{remark}

\section{General results for consensus stabilization under perturbed information}

As already pointed out in the previous section, the main strategy for studying under which assumptions a solution of system \eqref{eq:cuckersmale} tends to consensus is to obtain an estimate of the decay of the functional $V(t)$. We follow a similar approach in order to study consensus emergence for system  \eqref{eq:cuckersmale_perturbed}. We begin by proving the following lemma:

\begin{lemma} \label{lem:bigv_growth}
Let $(\bm{x}(t), \bm{v}(t))$ be a solution of system \eqref{eq:cuckersmale_perturbed}. For every $t \geq 0$, it holds
\begin{align} \label{eq:maintool}
\frac{d}{dt} V(t) \leq - 2 a\left(\sqrt{2NX(t)}\right) V(t) -2 \alpha V(t) + \frac{2 \beta}{N} \sum^N_{i = 1} \scalarp{\Delta_i(t), \bm{v}^{\perp}_i(t)}\,.
\end{align}
\end{lemma}
\begin{proof}
Differentiating $V$ for every $t \geq 0,$ we have
\begin{eqnarray*}
	\frac{d}{dt} V(t) & = &  \frac{d}{dt} \frac{1}{N} \sum^N_{i = 1} \vnorm{\bm{v}^{\perp}_i(t)}^2 \\
	& = & \frac{2}{N} \sum^N_{i = 1} \scalarp{\frac{d}{dt} \bm{v}^{\perp}_i(t), \bm{v}^{\perp}_i(t)} \\
	& = & \frac{2}{N} \sum^N_{i = 1} \scalarp{\frac{d}{dt} \bm{v}_i(t), \bm{v}^{\perp}_i(t)} - \frac{2}{N} \sum^N_{i = 1} \scalarp{\frac{d}{dt} \overline{\bm{v}}(t), \bm{v}^{\perp}_i(t)},
\end{eqnarray*}
which, inserting the expression for $\frac{d}{dt} \bm{v}_i(t)$, yields
\begin{eqnarray} \label{eq:derivative_bigv}
\begin{split}
	\frac{d}{dt} V(t) & = \underbrace{\frac{2}{N^2} \sum^N_{i = 1} \sum^N_{j = 1} a\left(r_{ij}\right)\scalarp{\bm{v}_{j}(t)-\bm{v}_{i}(t), \bm{v}^{\perp}_i(t)}}_{(i)} + \\
	& \quad + \frac{2 \alpha}{N} \sum^N_{i = 1} \scalarp{\overline{\bm{v}}(t) - \bm{v}_i(t), \bm{v}^{\perp}_i(t)} + \frac{2 \beta}{N} \sum^N_{i = 1} \scalarp{\Delta_i(t), \bm{v}^{\perp}_i(t)} \\
	& \quad - \frac{2}{N} \sum^N_{i = 1} \scalarp{\frac{d}{dt} \overline{\bm{v}}(t), \bm{v}^{\perp}_i(t)}.
\end{split}
\end{eqnarray}
Since
\begin{eqnarray} \label{eq:distance_est}
\begin{split}
r_{ij} = \vnorm{\bm{x}_i - \bm{x}_j} & = \vnorm{\bm{x}_i^{\perp} - \bm{x}_j^{\perp}} \\
& \leq \vnorm{\bm{x}_i^{\perp}} + \vnorm{\bm{x}_j^{\perp}} \\
& \leq \sqrt{2} \left( \sum^N_{k = 1} \vnorm{\bm{x}_k^{\perp}}\right)^{\frac{1}{2}} \\
& \leq \sqrt{2 N X},
\end{split}
\end{eqnarray}
the fact that $a$ is non increasing and inequality \eqref{eq:maintrick} yield
\begin{align} \label{eq:lemma1_CFPT}
(i) \leq - 2 a\left(\sqrt{2NX(t)}\right) V(t).
\end{align}
Using Proposition \ref{prop:derivative_mean}, we can rewrite the remaining term as
\begin{equation} \label{eq:controlterm}
\begin{split}
	\frac{d}{dt} V(t) - (i) & = -\frac{2 \alpha}{N} \sum^N_{i = 1} \vnorm{\bm{v}^{\perp}_i(t)}^2 + \frac{2 \beta}{N} \sum^N_{i = 1} \scalarp{\Delta_i(t), \bm{v}^{\perp}_i(t)} - \underbrace{\frac{2\beta}{N} \sum^N_{i = 1} \scalarp{\overline{\Delta}(t), \bm{v}^{\perp}_i(t)}}_{= 0}\\
	& = -2 \alpha V(t) + \frac{2 \beta}{N} \sum^N_{i = 1} \scalarp{\Delta_i(t), \bm{v}^{\perp}_i(t)}
\end{split}
\end{equation}
Applying \eqref{eq:lemma1_CFPT} and \eqref{eq:controlterm} on \eqref{eq:derivative_bigv} concludes the proof.
\end{proof}
As a direct consequence we obtain the following theorem.

\begin{theorem} \label{th:perpbound_convergence}
Let $(\bm{x}(t), \bm{v}(t))$ be a solution of system \eqref{eq:cuckersmale_perturbed}, and suppose that there exists a $T \geq 0$ such that for every $t \geq T$,
\begin{align} \label{eq:smallerror}
\sum^N_{i = 1} \scalarp{\Delta_i(t), \bm{v}^{\perp}_i(t)} \leq \phi(t) \sum^N_{i = 1} \vnorm{\bm{v}^{\perp}_i(t)}^2
\end{align}
for some function $\phi:[T,+\infty) \rightarrow [0,\ell]$, where $\ell<\frac{\alpha}{\beta}$. Then $(\bm{x}(t), \bm{v}(t))$ tends to consensus.
\end{theorem}
\begin{proof}
Under the assumption \eqref{eq:smallerror}, for every $t \geq T$ the upper bound in \eqref{eq:maintool} can be simplified to
\begin{equation*}
\begin{split}
\frac{d}{dt}V(t) & \leq -2 \alpha V(t) + \frac{2 \beta}{N} \sum^N_{i = 1} \scalarp{\Delta_i(t), \bm{v}^{\perp}_i(t)} \\
& \leq -2 \alpha V(t) + \frac{2 \beta}{N} \phi(t) \sum^N_{i = 1} \vnorm{\bm{v}^{\perp}_i(t)}^2 \\
& = -2 \alpha V(t) + 2 \beta \phi(t) V(t) \\
& \leq 2 \beta \left(\ell - \frac{\alpha}{\beta}\right) V(t).
\end{split}
\end{equation*}
Integrating between $T$ and $t$ (where $t \geq T$) we get
\begin{equation*}
V(t) \leq V(T) e^{2 \beta \left(\ell - \frac{\alpha}{\beta}\right)(t - T)}
\end{equation*}
and as the factor $\ell - \frac{\alpha}{\beta}$ is negative, $V(t)$ approaches $0$ exponentially fast.
\end{proof}

\begin{corollary} \label{cor:noperp_convergence}
If there exists $T \geq 0$ such that $\Delta^{\perp}_i(t) = 0$ for every $t \geq T$ and for every $1 \leq i \leq N$, then any solution of system \eqref{eq:cuckersmale_perturbed} tends to consensus.
\end{corollary}
\begin{proof}
Noting that $\Delta^{\perp}_i = 0$ implies $\Delta_i = \overline{\Delta}$, we have, by \eqref{eq:vertequalzero}
\begin{align*}
\sum^N_{i = 1} \scalarp{\Delta_i(t), \bm{v}^{\perp}_i(t)} = \sum^N_{i = 1} \scalarp{\overline{\Delta}, \bm{v}^{\perp}_i(t)} = 0.
\end{align*}
We apply Theorem \ref{th:perpbound_convergence} with $\phi(t) \equiv 0$ for every $t \geq T$ and obtain the result.
\end{proof}

\begin{remark}
Corollary \ref{cor:noperp_convergence} implies trivially that any solution of system \eqref{eq:cuckersmale_uniform} tends to consensus (this was already a consequence of Theorem \ref{thm:hhk}), but has moreover a rather nontrivial implication: also any solution of systems subjected to \emph{deviated} uniform control, i.e., systems like \eqref{eq:cuckersmale_perturbed} where $\Delta_i(t) = \Delta(t)$ for every $1 \leq i \leq N$ and for every $t \geq 0$, tend to consensus, because for every $1 \leq i \leq N$ and for every $t \geq 0$, we have
\begin{align*}
\begin{split}
\Delta^{\perp}_i(t) = \Delta_i(t) - \frac{1}{N}\sum^N_{j = 1} \Delta_j(t) = \Delta(t) - \Delta(t) = 0,
\end{split}
\end{align*}
and thus Corollary \ref{cor:noperp_convergence} applies. This means that systems of this kind converge to consensus even if the agents have an incorrect knowledge of the mean velocity, provided they all possess the same deviation. 
\end{remark}
A final consequence of the previously developed results is the following theorem, which provides  an upper bound for tolerable perturbations under which consensus emergence can be unconditionally guaranteed.
\begin{theorem}
Let $\varepsilon_i: [0, +\infty) \rightarrow [0,\ell]$ for a fixed $\ell < \frac{\alpha}{\beta}$, for every $i = 1, \ldots, N$. If there exists $T \geq 0$ such that $\vnorm{\Delta_i(t)} \leq \varepsilon_i(t) \vnorm{\bm{v}_i^{\perp}(t)}$ for every $t \geq T$ and for every $1 \leq i \leq N$, then any solution of system \eqref{eq:cuckersmale_perturbed} tends to consensus.
\end{theorem}
\begin{proof}
By using the Cauchy-Schwarz inequality we have
\begin{equation*}
\begin{split}
\sum^N_{i = 1} \scalarp{\Delta_i(t), \bm{v}^{\perp}_i(t)} & \leq \sum^N_{i = 1} \varepsilon_i(t) \vnorm{\bm{v}^{\perp}_i(t)}^2 \\
& \leq \ell \sum^N_{i = 1} \vnorm{\bm{v}^{\perp}_i(t)}^2. \\
\end{split}
\end{equation*}
The conclusion follows by taking $\phi(t) \equiv \ell$ for $t \in [T, +\infty)$ in Theorem \ref{th:perpbound_convergence}.
\end{proof}

\begin{remark}
The result above shows that, provided that the magnitude of the perturbation is smaller than the one of the deviation of the agent velocity from the mean, then convergence to consensus is obtained unconditionally with respect to the initial condition. This is the case of local estimations of the average, as the largest error that an agent can make when estimating the group average upon a subset of agents is precisely its own deviation from the mean, $\bm{v}^{\perp}_i$.
\end{remark}

\section{Perturbations as linear combinations of velocity deviations} \label{sec:cuckersmale_localmean}

We begin this section by considering a simple case study, which is nevertheless relevant as it addresses consensus stabilization based on a leader-following feedback.

Let us use Lemma \ref{lem:bigv_growth} to study the convergence to consensus of a system like \eqref{eq:cuckersmale_local}, where each agent computes its local mean velocity $\overline{\bm{v}}_i$ by taking into account themselves and only a single common agent $(x_1, v_1)$, which in turn takes into account only itself by computing $\overline{\bm{v}}_1 = \bm{v}_1$. Formally, given two finite conjugate exponents $p, q$ (i.e., two positive real numbers satisfying $\frac{1}{p} + \frac{1}{q} = 1$), we assume that for any $i = 1, \ldots, N$
\begin{align*}
\overline{\bm{v}}_i(t) = \frac{1}{p}\bm{v}_i(t) + \frac{1}{q}\bm{v}_1(t).
\end{align*}
We shall prove that any solution of this system tends to consensus, no matter how small the weight $\frac{1}{q}$ of $\bm{v}_1$ in $\overline{\bm{v}}_i$ is. We start by writing the system under the form \eqref{eq:cuckersmale_perturbed}, with $\alpha = \beta = \gamma$ and
\begin{align*}
\Delta_i(t) = \frac{1}{p}\bm{v}^{\perp}_i(t) + \frac{1}{q}\bm{v}^{\perp}_1(t).
\end{align*}
Hence, the perturbation term in \eqref{eq:maintool} is thus
\begin{equation*}
\begin{split}
	\frac{2 \gamma}{N}\sum^N_{i = 1} \scalarp{\Delta_i(t), \bm{v}^{\perp}_i(t)} & =  \frac{2 \gamma}{N}\sum^N_{i = 1} \scalarp{\frac{1}{p}\bm{v}^{\perp}_i(t) + \frac{1}{q}\bm{v}^{\perp}_1(t), \bm{v}^{\perp}_i(t)}\\
	& = \frac{1}{p} \frac{2 \gamma}{N} \sum^N_{i = 1} \vnorm{\bm{v}_i^{\perp}(t)}^2 + \frac{1}{q} \frac{2 \gamma}{N} \scalarp{\bm{v}^{\perp}_1(t), \underbrace{\sum^N_{i = 1} \bm{v}^{\perp}_i(t)}_{= 0}} \\
	& = \frac{2 \gamma}{p} V(t)\,,
\end{split}
\end{equation*}
and Lemma \ref{lem:bigv_growth} let us bound the growth of $V(t)$ as
\begin{align*}
	\frac{d}{dt}V(t) \leq 2 \gamma \left(-1 + \frac{1}{p}\right) V(t) = -\frac{2 \gamma}{q} V(t)\,.
\end{align*}
This ensures the exponential decay of the functional $V(t)$ for any $q > 0$.

Motivated by the latter configuration,  we turn our attention to the study of systems like \eqref{eq:cuckersmale_perturbed} where the perturbation of the mean of the $i$-th agent has the form
\begin{align} \label{eq:pert_cucker}
\Delta_i(t) = \sum^N_{j = 1} \omega_{ij}(t) \bm{v}^{\perp}_j(t).
\end{align}
We shall see that the results obtained in Section \ref{sec:first_results} help us identify under which assumptions on the coefficients $\omega_{ij}$ we can infer unconditional convergence to consensus.

\begin{theorem} \label{th:suff_cond_consensus}
Consider a system of the form \eqref{eq:cuckersmale_perturbed}, where $\Delta_i$ is given as in \eqref{eq:pert_cucker}. Then, if for every $t \geq 0$ and every $i,j = 1, \ldots, N$ we have $\omega_{ij}(t) = \omega_{ji}(t)$, and we set
\begin{align*}
I(t) := \min_{i,j} \omega_{ij}(t) \text{ and } S(t) := \max_{i} \sum^N_{j = 1}\omega_{ij}(t),
\end{align*}
the following estimate holds:
\begin{align} \label{eq:decayperp}
\frac{d}{dt}V(t) \leq - 2 a\left(\sqrt{2NX(t)}\right) V(t) + 2\beta \left( S(t) - NI(t) - \frac{\alpha}{\beta} \right) V(t).
\end{align}
Therefore, if there exists a $T \geq 0$ such that the quantity $S(t) - NI(t) - \frac{\alpha}{\beta}$ is bounded from above by a constant $C < 0$ in $[T, +\infty)$, then any solution of the system tends to consensus.
\end{theorem}
\begin{proof}
Standard calculations yield
\begin{equation*}
\begin{split}
\sum^N_{i = 1} \scalarp{\Delta_i(t), \bm{v}^{\perp}_i(t)} & = \sum^N_{i = 1} \sum^N_{j = 1} \omega_{ij}(t) \scalarp{\bm{v}^{\perp}_j(t), \bm{v}^{\perp}_i(t)}\\
& = \sum^N_{i = 1} \sum^N_{j = 1} \omega_{ij}(t) \scalarp{\bm{v}^{\perp}_j(t) - \bm{v}^{\perp}_i(t), \bm{v}^{\perp}_i(t)} \\
& \quad + \sum^N_{i = 1} \left( \sum^N_{j = 1} \omega_{ij}(t) \right) \vnorm{\bm{v}^{\perp}_i(t)}^2\\
& = N^2 \left( -\frac{1}{2N^2} \sum^N_{i = 1} \sum^N_{j = 1} \omega_{ij}(t) \vnorm{\bm{v}_j(t) - \bm{v}_i(t)}^2\right) \\
& \quad + \sum^N_{i = 1} \left( \sum^N_{j = 1} \omega_{ij}(t) \right) \vnorm{\bm{v}^{\perp}_i(t)}^2\\
& \leq N \left(- NI(t)  + S(t) \right) V(t),
\end{split}
\end{equation*}
having used the equality \eqref{eq:maintrick}. Applying this into \eqref{eq:maintool} and collecting $\beta$, we get  \eqref{eq:decayperp}.
\end{proof}

In the following results we shall assume $\sum^N_{j = 1} \omega_{ij}(t) = 1$, which implies by Proposition \ref{prop:derivative_mean} that $\overline{\Delta} =0$ and that $\overline{\bm{v}}$ is conserved. In particular, for $\alpha = \beta$ the system can be eventually rewritten as a Cucker-Smale system of the
type \eqref{eq:cuckersmale}, with a different interaction function $a$, and the following results can be seen as consequences of  Theorem \ref{thm:hhk}.

\begin{corollary}
Let, for any $t \geq 0$, $\omega(t) \in [0,+\infty)^{N \times N}$ be a symmetric stochastic matrix, i.e., $\sum^N_{j = 1} \omega_{ij}(t) = 1$ for every $i = 1, \ldots, N$. If there exists a $\vartheta > 0$ such that:
\begin{align*}
I(t) = \min_{i,j} \omega_{ij}(t) \geq \vartheta > \frac{\beta - \alpha}{N \beta},
\end{align*}
then any solution of the system tends to consensus.
\end{corollary}
\begin{proof}
Under the above hypotheses, the quantity $S(t) - NI(t) - \frac{\alpha}{\beta}$ of Theorem \ref{th:suff_cond_consensus} is bounded from above by $1 - N \vartheta - \frac{\alpha}{\beta}$, which, by assumption, is negative.
\end{proof}

\begin{corollary} \label{cor:phi_consensus}
Suppose that
\begin{align*}
\omega_{ij}(t) = \frac{\phi(r_{ij}(t))}{\eta(t)}
\end{align*}
where $\phi:[0, +\infty) \rightarrow \left(0,1\right]$ is a non increasing, positive, bounded function, and $\eta:[0, +\infty) \rightarrow \left[0,+\infty\right)$ is a nonnegative function. Then, given constants $\alpha, \beta \geq 0$ satisfying
\begin{align} \label{eq:betaalpha}
1 \leq N \frac{\beta}{\alpha} \leq \eta(t) \quad \text{for every } t \in [0,+\infty),
\end{align}
we have that any solution of system \eqref{eq:cuckersmale_perturbed}, for $\Delta_i$ as in \eqref{eq:pert_cucker} and $\alpha$ and $\beta$ as above, tends to consensus.
\end{corollary}
\begin{proof}
If we consider the quantity $X(t)$, there are at most two cases: either $X(t)$ is bounded from above by a constant $\overline{X}$ in $[0, +\infty)$, or $X(t)$ remains unbounded.

In the first case, we can  bound $S(t)$ from above by $N$. Since $r_{ij}(t) \leq \sqrt{2NX(t)} \leq \sqrt{2N\overline{X}}$ and $\phi$ is non increasing, we have that
\begin{align*}
I(t) = \min_{ij} \omega_{ij}(t) \geq \phi\left(\sqrt{2N\overline{X}}\right),
\end{align*}
and thus, from \eqref{eq:betaalpha}, it follows that
\begin{align} \label{eq:invoketheorem}
S(t) - N I(t) - \frac{\alpha}{\beta}\eta(t) \leq N - N \phi\left(\sqrt{2N\overline{X}}\right) - \frac{\alpha}{\beta}\eta(t) \leq - N \phi\left(\sqrt{2N\overline{X}}\right).
\end{align}
Since $\phi$ is positive, the proof is completed by using Theorem \ref{th:suff_cond_consensus}.

Suppose now, instead, that $X(t)$ is unbounded: in this case the term $I(t)$ is bounded from below by a term going to $0$ (and hence not helping us) and we have to take advantage of $S(t)$ as shall be shown now. By definition, $X(t)$ is unbounded if and only if there exist two agents with indexes $h$ and $k$ such that $r_{hk}(t)$ is unbounded. By the triangle inequality, it follows that for any index $i$ , there is an index $j(i)$ for which $r_{ij(i)}$ is unbounded. Thus, we fix $\rho > 0$ and let $T > 0$ be the maximum time $t$ such that $\phi(r_{ij(i)}(t)) < 1 - \rho$ for every $i = 1, \ldots, N$. Then for every $t > T$ we may bound $S(t)$ as
\begin{align*}
S(t) \leq \frac{1}{\eta(t)}\left(N - 1 + 1 - \rho \right) \leq \frac{\alpha}{N \beta}(N - \rho)
\end{align*}
and therefore, again from \eqref{eq:betaalpha},
\begin{align*}
S(t) - N I(t) - \frac{\alpha}{\beta}\eta(t) \leq \frac{\alpha}{N \beta}(N - \rho) - \frac{\alpha}{\beta}\eta(t) \leq - \frac{\alpha}{N \beta}\rho + \frac{\alpha}{\beta} - N \leq - \frac{\alpha}{N \beta}\rho < 0
\end{align*}
holds for every $t \geq T$, since $N \frac{\beta}{\alpha} \geq 1$ by assumption. Theorem \ref{th:suff_cond_consensus} yields thus the result.
\end{proof}

\begin{remark} \label{re:corollary_phi}
A concrete example of a system for which we can apply Corollary \ref{cor:phi_consensus} is obtained by considering the functions
\begin{align*}
\phi(r) = \frac{1}{(1 + r^2)^{\epsilon}}\,,
\end{align*}
and
\begin{align*}
\eta(t) = \max_i\left\{\sum^N_{j = 1} \phi(r_{ij}(t))\right\}.
\end{align*}
In this case, $\epsilon$ can be thought as a parameter tuning the ability of each particle to gather information about the speed of the other agents: indeed, consider a set of agents such that $r_{ij} > 0$, if $i \not = j$. Then, if $\epsilon$ is $0$, $\overline{\bm{v}}_i = \overline{\bm{v}}$ for every $i$, and each particle communicates at the same rate with near and far away agents, while if $\epsilon$ goes to $+\infty$ then $\overline{\bm{v}}_i$ approaches $\bm{v}_i$, hence each particle is unable to gain knowledge about the speed of the other agents. The function $\eta$ serves to the purpose of being a common normalizing factor: naturally one would choose for every agent $i$ the normalizing factor given by
\begin{align} \label{eq:usual_choice}
\sum^N_{j = 1} \phi(r_{ij}(t)),
\end{align}
but that would produce a non symmetric matrix $\omega$, for which the above results are not valid. In this context, the function $\eta$ is a suitable replacement, being also coherent with the asymptotic behavior of \eqref{eq:usual_choice} for $\epsilon \rightarrow 0$ and $\epsilon \rightarrow +\infty$.
\end{remark}

\begin{remark} \label{rem:functionR}
The request of positivity of the function $\phi$ cannot be removed from Corollary \ref{cor:phi_consensus}, as the function
\begin{align*}
\phi(r) = \chi_{[0,R]}(r) = \left\{
\begin{array}{ll}
1 & \text{ if } r \leq R \\
0 & \text{ if } r > R \\
\end{array}
\right.
\end{align*}
shows.
Indeed, what fails in the argument of the proof is the case in which we suppose that $X(t)$ is bounded by $\overline{X}$: if the quantity $\sqrt{2N\overline{X}}$ is not less or equal to $R$, then $\phi\left(\sqrt{2N\overline{X}}\right) = 0$ in the inequality \eqref{eq:invoketheorem}, and we cannot invoke Theorem \ref{th:suff_cond_consensus} in order to infer consensus.
\end{remark}

\section{Perturbations due to local averaging} \label{sec:localmeanR}

An interesting case of a system like \eqref{eq:cuckersmale_local} is the one where the local mean is calculated as
\begin{align*}
\overline{\bm{v}}_i = \frac{1}{\# \Lambda_R(i)} \sum_{j \in \Lambda_R(i)} \bm{v}_j,
\end{align*}
where $\Lambda_R(i) = \left\{j \in \{1, \ldots, N\} \mid r_{ij} \leq R \right\}$ and $\# \Lambda_R(i)$ is its cardinality. In this case, we model the situation in which each agent calculates its local mean counting only those agents inside a ball of radius $R$ centered on him. We want to address the issue of characterizing the behavior of system \eqref{eq:cuckersmale_local} with the above choice for $\overline{\bm{v}}_i$ when the radius $R$ of each ball is either reduced to $0$ or set to grow to $+\infty$: we shall see that we can reformulate  this decentralized system again as a Cucker-Smale model for a different interaction function for which we can apply Theorem \ref{thm:hhk}. We shall show how tuning the radius $R$ affects the convergence to consensus, from the case $R \geq 0$ where only conditional convergence is ensured, to the unconditional convergence result given for $R = +\infty$.

First of all, by defining $\chi_{[0,R]}(r)$ as the characteristic function of a ball of radius $R$ centered at the origin,  we can rewrite $\overline{\bm{v}}_i$ as
\begin{align} \label{eq:truncated_perturbation}
\overline{\bm{v}}_i = \frac{1}{\sum^N_{k = 1} \chi_{[0,R]}(r_{ik})} \sum_{j = 1}^N \chi_{[0,R]}(r_{ij}) \bm{v}_j.
\end{align}
As already noted in Remark \ref{re:corollary_phi}, the normalizing terms $\sum^N_{k = 1} \chi_{[0,R]}(r_{ik}(t))$ give rise to a matrix of weights which is not symmetric. Since this will be an issue also in the present section, we take $\eta_R(t)$ to be a function approximating the above normalizing terms and which also preserves its asymptotics for $R \rightarrow 0$ and $R \rightarrow +\infty$, as for instance,
\begin{align} \label{eq:normalizingR}
\eta_R(t) = \max_i \left\{\sum^N_{k = 1} \chi_{[0,R]}(r_{ik}(t)) \right\}.
\end{align}
For every $t \geq 0$, we replace the vector $\overline{\bm{v}}_i(t) $ by
\begin{align*}
\frac{1}{\eta_{R}(t)} \sum^N_{j  = 1} \chi_{[0,R]}(r_{ij}(t)) \bm{v}_j(t)\,.
\end{align*}
Moreover, the vector
\begin{align*}
\bm{v}_i(t) \cdot \left(\frac{1}{\eta_{R}(t)} \sum^N_{j  = 1} \chi_{[0,R]}(r_{ij}(t)) \right)
\end{align*}
is also an approximation of $\bm{v}_i(t)$ for $R \rightarrow 0$ and $R \rightarrow +\infty$. This motivates the replacement of the term $\overline{\bm{v}}_i - \bm{v}_i$ in system \eqref{eq:cuckersmale_local} where $\overline{\bm{v}}_i$ is as in \eqref{eq:truncated_perturbation}, with
\begin{equation}
\begin{split} \label{eq:derivationR}
\frac{1}{\eta_{R}} \sum^N_{j  = 1} \chi_{[0,R]}(r_{ij})\bm{v}_j - \left(\frac{1}{\eta_{R}}  \sum^N_{j  = 1} \chi_{[0,R]}(r_{ij}) \right)\bm{v}_i & =\frac{1}{\eta_{R}}  \sum^N_{j  = 1} \chi_{[0,R]}(r_{ij}) (\bm{v}_j - \bm{v}_i) \\
& = \frac{1}{\eta_{R}}  \sum^N_{i = 1} (\bm{v}_j - \bm{v}_i) \\
& \quad + \frac{1}{\eta_{R}}  \sum^N_{j  = 1} (1 - \chi_{[0,R]}(r_{ij})) (\bm{v}_i - \bm{v}_j) \\
& = \frac{N}{\eta_{R}} \left(\overline{\bm{v}} - \bm{v}_i\right) \\
& \quad + \frac{1}{\eta_{R}}  \sum^N_{j  = 1} (1 - \chi_{[0,R]}(r_{ij})) (\bm{v}_i - \bm{v}_j).
\end{split}
\end{equation}

We can thus rewrite the original system as \eqref{eq:cuckersmale_perturbed}:
\begin{eqnarray}
\left\{
\begin{aligned}
\begin{split} \label{eq:cuckersmale_R}
\dot{\bm{x}}_{i} & = \bm{v}_{i} \\
\dot{\bm{v}}_{i} & = \frac{1}{N} \sum_{j = 1}^N a\left(r_{ij}\right)\left(\bm{v}_{j}-\bm{v}_{i}\right) + \gamma \frac{N}{\eta_{\varepsilon}} (\overline{\bm{v}} - \bm{v}_i) + \gamma \Delta^{\varepsilon}_i,
\end{split}
\end{aligned}
\right.
\end{eqnarray}
where the perturbations have the general form
\begin{align} \label{eq:pert_repulsion}
\Delta^{\varepsilon}_i(t) = \frac{1}{\eta_{\varepsilon}(t)} \sum^N_{j = 1} (1 - \psi_{\varepsilon}(r_{ij}(t))) (\bm{v}_i(t) - \bm{v}_j(t)).
\end{align}
For \eqref{eq:pert_repulsion} to be a coherent approximation of our case study, we prescribe that $\varepsilon$ is a parameter ranging in a nonempty set $\Omega$ satisfying:
\begin{enumerate}[$(i)$]
\item \label{item:i} $\psi_{\varepsilon}:[0, +\infty) \rightarrow [0,1]$ is a nonincreasing measurable function for every $\varepsilon \in \Omega$;
\item \label{item:ii} $\eta_{\varepsilon}:[0, +\infty) \rightarrow \mathbb{R}$ is an $L^{\infty}$-function for every $\varepsilon \in \Omega$;
\item \label{item:iii} there are two disjoint subsets $\Omega_{C\!S}$ and $\Omega_{U}$ of $\Omega$ such that, if $\varepsilon \in \Omega_{C\!S}$ then $\psi_{\varepsilon} \equiv \chi_{\{0\}}$ and $\eta_{\varepsilon} \equiv 1$, while if $\varepsilon \in \Omega_{U}$ then $\psi_{\varepsilon} \equiv \chi_{[0, +\infty)}$ and $\eta_{\varepsilon} \equiv N$.
\end{enumerate}
With requirement \eqref{item:iii}, we impose that if $\varepsilon \in \Omega_{C\!S}$ then $\Delta_i^{\varepsilon} = - \frac{N}{\eta_{\varepsilon}} (\overline{\bm{v}} - \bm{v}_i)$, therefore recovering system \eqref{eq:cuckersmale} from \eqref{eq:cuckersmale_R}, whereas if $\varepsilon \in \Omega_{U}$ then $\Delta_i^{\varepsilon} = 0$, and we obtain a particular instance of system \eqref{eq:cuckersmale_uniform}.

In order to study under which conditions on the initial values the solutions of system \eqref{eq:cuckersmale_R} converge to consensus, we cannot use the results of the previous section. This is because, if we compare the following calculations
\begin{equation*}
\begin{split}
\frac{1}{\eta_{R}} \sum^N_{j  = 1} \chi_{[0,R]}(r_{ij})\bm{v}_j - \left(\frac{1}{\eta_{R}}  \sum^N_{j  = 1} \chi_{[0,R]}(r_{ij}) \right)\bm{v}_i & = \left(\frac{1}{\eta_{R}}  \sum^N_{j  = 1} \chi_{[0,R]}(r_{ij}) \right)(\overline{\bm{v}} - \bm{v}_i) \\
& \quad + \frac{1}{\eta_{R}} \sum^N_{j  = 1} \chi_{[0,R]}(r_{ij})\bm{v}^{\perp}_j
\end{split}
\end{equation*}
to \eqref{eq:derivationR}, we obtain that the system we are considering is precisely the version of system \eqref{eq:cuckersmale_perturbed} where $\alpha = \frac{1}{\eta_{R}}  \sum^N_{j  = 1} \chi_{[0,R]}(r_{ij})$ (which is equal to 1 for sufficiently small and large enough $R$) and $\Delta_i$ is of the same kind as the one mentioned in Remark \ref{rem:functionR}.

We present the following result which gives a sufficient condition on the initial data for which the solutions of system \eqref{eq:cuckersmale_R} converge to consensus. We point out that system \eqref{eq:cuckersmale_R} can be rewritten into the shape of a Cucker-Smale type of model as follows:
\begin{eqnarray*}
\left\{
\begin{aligned}
\begin{split}
\dot{\bm{x}}_{i} & = \bm{v}_{i} \\
\dot{\bm{v}}_{i} & = \frac{1}{N} \sum_{j = 1}^N \left(a\left(r_{ij}\right) + \gamma \frac{N}{\eta_{\varepsilon}}\psi_{\varepsilon}(r_{ij})\right) (\overline{\bm{v}} - \bm{v}_i),
\end{split}
\end{aligned}
\right.
\end{eqnarray*}
and therefore, the following result is obtained as an application of Theorem \ref{thm:hhk}.

\begin{theorem} \label{th:HaHaKimExtended}
Fix $\gamma \geq 0$, consider system \eqref{eq:cuckersmale_R} where $\Delta^{\varepsilon}_i$ is as in \eqref{eq:pert_repulsion} and let $(\bm{x}_0,\bm{v}_0) \in (\mathbb{R}^d)^N \times (\mathbb{R}^d)^N$. Then if $X_0 = B(\bm{x}_0, \bm{x}_0)$ and $V_0 = B(\bm{v}_0,\bm{v}_0)$ satisfy
\begin{align} \label{eq:HaHaKimLarge}
\int^{+\infty}_{\sqrt{X_0}} a\left(\sqrt{2N}r\right) \ dr + \frac{\gamma N}{\vnorm{\eta_{\varepsilon}}_{\infty}} \int^{+\infty}_{\sqrt{X_0}} \psi_{\varepsilon}\left(\sqrt{2N}r\right) \ dr \geq \sqrt{V_0},
\end{align}
the solution of system \eqref{eq:cuckersmale_R} with initial datum $(\bm{x}_0,\bm{v}_0)$ tends to consensus.
\end{theorem}
\begin{proof}
From requirement \eqref{item:i} we have
\begin{equation*}
\begin{split}
\frac{1}{N}\sum^N_{i = 1} \scalarp{\Delta_i(t), \bm{v}^{\perp}_i(t)} & = \frac{1}{N\eta_{\varepsilon}(t)} \sum^N_{i = 1} \sum^N_{j = 1} (1 - \psi_{\varepsilon}(r_{ij}(t))) \scalarp{\bm{v}_i(t) - \bm{v}_j(t), \bm{v}^{\perp}_i(t)}\\
& \leq  \frac{N}{\eta_{\varepsilon}(t)} \left(1 - \psi_{\varepsilon} \left(\sqrt{2N X(t)}\right)\right) V(t),
\end{split}
\end{equation*}
and inequality \eqref{eq:maintool} reads
\begin{align*}
\frac{d}{dt}V(t) \leq - 2  \left(a\left(\sqrt{2NX(t)}\right) + \frac{\gamma N}{\eta_{\varepsilon}(t)} \psi_{\varepsilon} \left(\sqrt{2N X(t)}\right) \right) V(t).
\end{align*}
Moreover, since
\begin{align*}
\frac{d}{dt}\sqrt{V(t)} = \frac{1}{2 \sqrt{V(t)}} \frac{d}{dt}V(t),
\end{align*}
from \eqref{item:ii} we have
\begin{equation}
\begin{split}\label{eq:HaHaDecay}
\frac{d}{dt}\sqrt{V(t)} & \leq - 2  \left(a\left(\sqrt{2NX(t)}\right) + \frac{\gamma N}{\eta_{\varepsilon}(t)} \psi_{\varepsilon} \left(\sqrt{2N X(t)}\right) \right) \sqrt{V(t)} \\
& \leq - 2  \left(a\left(\sqrt{2NX(t)}\right) + \frac{\gamma N}{\vnorm{\eta_{\varepsilon}}_{\infty}} \psi_{\varepsilon} \left(\sqrt{2N X(t)}\right) \right) \sqrt{V(t)},
\end{split}
\end{equation}
and integrating between $0$ and $t$ we obtain
\begin{align} \label{eq:growthbound}
\sqrt{V(t)} - \sqrt{V(0)} \leq - \int^t_0 \left(a\left(\sqrt{2NX(t)}\right) + \frac{\gamma N}{\vnorm{\eta_{\varepsilon}}_{\infty}} \psi_{\varepsilon} \left(\sqrt{2N X(t)}\right) \right) \sqrt{V(s)} \ ds.
\end{align}
We now work on changing the variable inside the integral. We can actually claim that
\begin{align} \label{eq:changevariable}
\frac{d}{dt} X(t) \leq \sqrt{V(t)},
\end{align}
since, indeed, the following
\begin{equation*}
\begin{split}
\frac{d}{dt} X(t) = \frac{1}{N} \sum^N_{i = 1}  \frac{d}{dt}\vnorm{x^{\perp}_i(t)}^2 & = \frac{2}{N} \sum^N_{i = 1} \scalarp{\bm{x}^{\perp}_i(t),\frac{d}{dt}\bm{x}^{\perp}_i(t)} \\
& = \frac{2}{N} \sum^N_{i = 1} \scalarp{\bm{x}^{\perp}_i(t),\bm{v}^{\perp}_i(t)} \\
& \leq \frac{2}{N} \sum^N_{i = 1} \vnorm{\bm{x}^{\perp}_i(t)}\vnorm{\bm{v}^{\perp}_i(t)} \\
& \leq \frac{2}{N} \left(\sum^N_{i = 1}\vnorm{\bm{x}^{\perp}_i(t)}\right)^{\frac{1}{2}} \left(\sum^N_{i = 1} \vnorm{\bm{v}^{\perp}_i(t)}^2\right)^{\frac{1}{2}} \\
& = 2 \sqrt{X(t)}\sqrt{V(t)},
\end{split}
\end{equation*}
and $\frac{d}{dt} X(t) = \frac{d}{dt}(\sqrt{X(t)}\sqrt{X(t)}) = 2 \sqrt{X(t)} \frac{d}{dt} \sqrt{X(t)}$, together yield \eqref{eq:changevariable}. Note that we have used
\begin{align*}
\frac{d}{dt}\bm{x}^{\perp}_i(t) = \frac{d}{dt}\left(\bm{x}_i(t) - \frac{1}{N} \sum^N_{j = 1} \bm{x}_j(t)\right) = \bm{v}_i(t) - \frac{1}{N} \sum^N_{j = 1} \bm{v}_j(t) = \bm{v}^{\perp}_i(t).
\end{align*}
Setting $r = \sqrt{X(s)}$, and using \eqref{eq:changevariable} we can change variable in \eqref{eq:growthbound} as follows:
\begin{align} \label{eq:cesemo}
\sqrt{V(t)} - \sqrt{V(0)} \leq - \int^{\sqrt{X(t)}}_{\sqrt{X(0)}} \left(a\left(\sqrt{2N}r\right) + \frac{\gamma N}{\vnorm{\eta_{\varepsilon}}_{\infty}} \psi_{\varepsilon}\left(\sqrt{2N}r\right) \right) \ dr.
\end{align}

Let us suppose that \eqref{eq:HaHaKimLarge} is true, and note that $X(0) = X_0$ and $V(0) = V_0$. If $V(0) = 0$, then there is nothing to prove since we are already in consensus. If, instead, it is true that
\begin{align} \label{eq:almostdone}
0 < \sqrt{V(0)} \leq \int^{+\infty}_{\sqrt{X(0)}} \left(a\left(\sqrt{2N}r\right) + \frac{\gamma N}{\vnorm{\eta_{\varepsilon}}_{\infty}} \psi_{\varepsilon}\left(\sqrt{2N}r\right) \right) \ dr,
\end{align}
then there is a $\overline{X} > X(0)$ such that
\begin{align*}
\sqrt{V(0)} = \int^{\sqrt{\overline{X}}}_{\sqrt{X(0)}} \left(a\left(\sqrt{2N}r\right) + \frac{\gamma N}{\vnorm{\eta_{\varepsilon}}_{\infty}} \psi_{\varepsilon}\left(\sqrt{2N}r\right) \right) \ dr
\end{align*}
(having used the fact that, from \eqref{item:i}, the integrand is a non increasing function). Now, either equality holds in \eqref{eq:almostdone}, and $\lim_{t \rightarrow +\infty} V(t) = 0$ follows by passing to the limit in \eqref{eq:cesemo}, or we have a strict inequality. But in this case $\overline{X} \geq X(t)$ must hold for every $t \geq 0$, since otherwise there would be a $T > 0$ for which we have
\begin{equation*}
\begin{split}
\sqrt{V(0)} & \geq \sqrt{V(T)} + \int^{\sqrt{X(T)}}_{\sqrt{X(0)}} \left(a\left(\sqrt{2N}r\right) + \frac{\gamma N}{\vnorm{\eta_{\varepsilon}}_{\infty}} \psi_{\varepsilon}\left(\sqrt{2N}r\right) \right) \ dr \\
& >  \int^{\sqrt{\overline{X}}}_{\sqrt{X(0)}} \left(a\left(\sqrt{2N}r\right) + \frac{\gamma N}{\vnorm{\eta_{\varepsilon}}_{\infty}} \psi_{\varepsilon}\left(\sqrt{2N}r\right) \right) \ dr \\
& = \sqrt{V(0)},
\end{split}
\end{equation*}
which is obviously a contradiction. Thus, we have that the inequality $\overline{X} \geq X(t)$ is true for every $t \geq 0$, and from \eqref{eq:HaHaDecay} we have
\begin{align*}
\frac{d}{dt}V(t) \leq - 2  \left(a\left(\sqrt{2N\overline{X}}\right) + \frac{\gamma N}{\vnorm{\eta_{\varepsilon}}_{\infty}} \psi_{\varepsilon} \left(\sqrt{2N \overline{X}}\right) \right) V(t).
\end{align*}
The fact that $\lim_{t \rightarrow +\infty} V(t) = 0$ follows from the inequality above.
\end{proof}

A first example of a family of functions $\left\{ \psi_{\varepsilon} \right\}_{\varepsilon \in \Omega}$ is given by
\begin{align*}
\psi_{\varepsilon}(r) = \frac{1}{(1 + r^2)^{\varepsilon}} \text{ where } \varepsilon \in \Omega = [0, \infty],
\end{align*}
for which we set $\psi_{\infty} \equiv \chi_{\{0\}}$ and
\begin{align*}
\eta_{\varepsilon}(t) = \max_i \left\{ \sum^N_{k = 1} \psi_{\varepsilon}(r_{ik}(t)) \right\}.
\end{align*}
In this case, $\Omega_{C\!S} = \{0\}$ and $\Omega_{U} = \{\infty\}$, and \eqref{eq:HaHaKimLarge} is satisfied as soon as
\begin{align*}
\int^{+\infty}_{\sqrt{X_0}} a\left(\sqrt{2N}r\right) \ dr + \gamma \int^{+\infty}_{\sqrt{X_0}} \psi_{\varepsilon}\left(\sqrt{2N}r\right) \ dr \geq \sqrt{V_0}
\end{align*}
is, since $\vnorm{\eta_{\varepsilon}}_{\infty} \leq N$.

But the most interesting example of such a family is the one which has introduced this section: we consider $\Omega = [0, +\infty]$, the set of functions $\{\chi_{[0,R]}\}_{R \in \Omega}$ and $\eta_R$ as in \eqref{eq:normalizingR} (notice that, as before, we have $\Omega_{C\!S} = \{0\}$ and $\Omega_{U} = \{\infty\}$). Since $\vnorm{\eta}_{\infty} \leq N$, if $R$ is sufficiently large to satisfy $\sqrt{2NX_0} \leq R$, condition \eqref{eq:HaHaKimLarge} is satisfied as soon as
\begin{align*}
\int^{+\infty}_{\sqrt{X_0}} a\left(\sqrt{2N}r\right) \ dr + \gamma \left(\frac{R}{\sqrt{2N}} - \sqrt{X_0}\right) \geq \sqrt{V_0},
\end{align*}
by means of a trivial integration. If, instead, $R$ is so small that $\sqrt{2NX_0} > R$, condition \eqref{eq:HaHaKimLarge} is satisfied as soon as
\begin{align*}
\int^{+\infty}_{\sqrt{X_0}} a\left(\sqrt{2N}r\right) \ dr \geq \sqrt{V_0},
\end{align*}
recovering Theorem \ref{thm:hhk}.

The above results can be seen as the asymptotic outcome of the following more general approach: consider the set $\Omega = [0,\infty] \times (1,\infty]$, write $\varepsilon$ as the couple of parameter $R, \theta$ and set
\begin{align*}
\psi_{R,\theta}(r) = \left\{
\begin{array}{ll}
1 & \text{ if } r \leq R, \\
\frac{1}{(r - R + 1)^{\theta}} & \text{ if } r > R. \\
\end{array}
\right.
\end{align*}
This time we have $\Omega_{C\!S} = \{0\}\times\{+ \infty\}$ and $\Omega_{U} = \{+\infty\} \times (1,+\infty]$.
If we suppose that $R$ is sufficiently large to satisfy $\sqrt{2NX_0} \leq R$ and we consider $\eta_{R, \theta}$ to be like
\begin{align*}
\eta_{R, \theta}(t) = \max_i \left\{\sum^N_{k = 1} \psi_{R, \theta}(r_{ik}(t)) \right\} \quad \text{or} \quad \eta_{R, \theta}(t) = \min_i \left\{\sum^N_{k = 1} \psi_{R, \theta}(r_{ik}(t))\right\},
\end{align*}
since in both cases we have $\vnorm{\eta}_{\infty} \leq N$, and it holds
\begin{align*}
\int^{+\infty}_{\sqrt{X_0}} \psi_{R, \theta}\left(\sqrt{2N}r\right) \ dr & = \int^{\frac{R}{\sqrt{2N}}}_{\sqrt{X_0}} \ dr + \int^{+\infty}_{\frac{R}{\sqrt{2N}}} \frac{11}{(\sqrt{2N}r - R + 1)^{\theta}} \ dr\\
& = \frac{R}{\sqrt{2N}} - \sqrt{X_0} + \frac{1}{\theta - 1},
\end{align*}
then condition \eqref{eq:HaHaKimLarge} is satisfied as soon as
\begin{align*}
\int^{+\infty}_{\sqrt{X_0}} a\left(\sqrt{2N}r\right) \ dr + \gamma \left( \frac{R}{\sqrt{2N}} - \sqrt{X_0} + \frac{1}{\theta - 1}\right)  \geq \sqrt{V_0},
\end{align*}
which shows that the consensus region grows linearly with the radius $R$ while it is inversely proportional to the growth of $\theta$.

Otherwise, if $R$ is so small that $\sqrt{2NX_0} > R$, since
\begin{align*}
\int^{+\infty}_{\sqrt{X_0}} \psi_{R, \theta}\left(\sqrt{2N}r\right) \ dr & = \int^{+\infty}_{\sqrt{X_0}} \frac{1}{(\sqrt{2N}r - R + 1)^{\theta}} \ dr\\
& = \frac{1}{(\theta - 1)(\sqrt{2NX_0} - R + 1)^{\theta - 1}}\,,
\end{align*}
we have that condition \eqref{eq:HaHaKimLarge} is true whenever
\begin{align*}
\int^{+\infty}_{\sqrt{X_0}} a\left(\sqrt{2N}r\right) \ dr + \frac{\gamma}{\theta - 1} \frac{1}{(\sqrt{2NX_0} - R + 1)^{\theta - 1}}  \geq \sqrt{V_0}\,.
\end{align*}
In this case, the consensus region is in practice not modified by $R$, but only by the decay of the far-away interaction $\theta$ (the faster the decay, the smaller the consensus region).

In both cases, if $\theta$ goes to $+\infty$, we recover the result we obtained for the family of functions $\{\chi_{[0,R]}\}_{R \in [0,+\infty]}$.

\section{Numerical tests}
We present a series of numerical tests illustrating the main results developed throughout this article. We begin by describing the generic setting upon which the initial configurations of agents are determined; we follow similar ideas as those presented in \cite{cafotove10}. We consider a system of $N$, $2-$dimensional agents with a randomly generated initial configuration of positions and velocities \[(\tilde{\bm{x}},\tilde{\bm{v}})\in[-1,1]^{2N}\times[-1,1]^{2N}\,,\]
interacting by means of the kernel \eqref{eq:kernel} , with $\delta=1$.
We recall that relevant quantities for the analysis of our results are given by
\[
X[\bm{x}](t)=\frac{1}{2N^2}\sum_{i,j,=1}^N||\bm{x}_i(t)-\bm{x}_j(t)||^2\,,\quad\text{and}\quad V[\bm{v}](t)=\frac{1}{2N^2}\sum_{i,j,=1}^N||\bm{v}_i(t)-\bm{v}_j(t)||^2\,.
\]
In fact, once a random initial configuration has been generated, it is possible to rescale it to a desired $(X_0,V_0)$ parametric pair, by means of
\[
(\bm{x},\bm{v})=\left(\sqrt{\frac{X_0}{X[\tilde{\bm{x}}]}}\tilde{\bm{x}}, \sqrt{\frac{V_0}{V[\tilde{\bm{v}}]}} \tilde{\bm{v}}\right)\,,
\]
such that $(X[\bm{x}],V[\bm{v}])=(X_0,V_0)$. As simulations concerning flock trajectories have been generated by prescribing a value for the pair $(X_0,V_0)$, which is used to rescale randomly generated initial conditions, there are slight variations on the initial positions and velocities in every model run, which can affect the final consensus direction. However our results are stated in terms of $X,V$, and independently of the specific initial configuration. For simulation purposes  the system is integrated in time with the specific feedback controller by means of a Runge-Kutta 4th-order scheme.

\noindent \textbf{Leader-based feedback.} The first case that we address is the one presented in Section \ref{sec:cuckersmale_localmean}, where the feedback is built upon local information and a single flock leader. In this case, the local feedback is defined as
\[\bm{u}_i=-(\bm{v}_i-\overline{\bm{v}}_i)\,,\quad\text{with}\quad \overline{\bm{v}}_i=(1-q)\bm{v}_i+q \bm{v}_1\,,\;i=1,\ldots,N\,,\]
where for convenience we have selected the first agent as the leader of the flock. Figure \ref{fig:1} shows the behavior of the flock depending on the parameter $q$, which represents the influence of the leader in the local average. Our result asserts that for $0< q\leq 1$, the system will converge to consensus independently of the initial configuration, which is illustrated by our numerical experiments, as shown in Figures \ref{fig:1} and \ref{fig:12} ; it can be observed that, the weaker the influence of the leader, the longer the flock takes to reach consensus.
\begin{figure}[!ht]
\centering
\resizebox{\textwidth}{!}{
\begin{tabular}{cc}
\epsfig{file=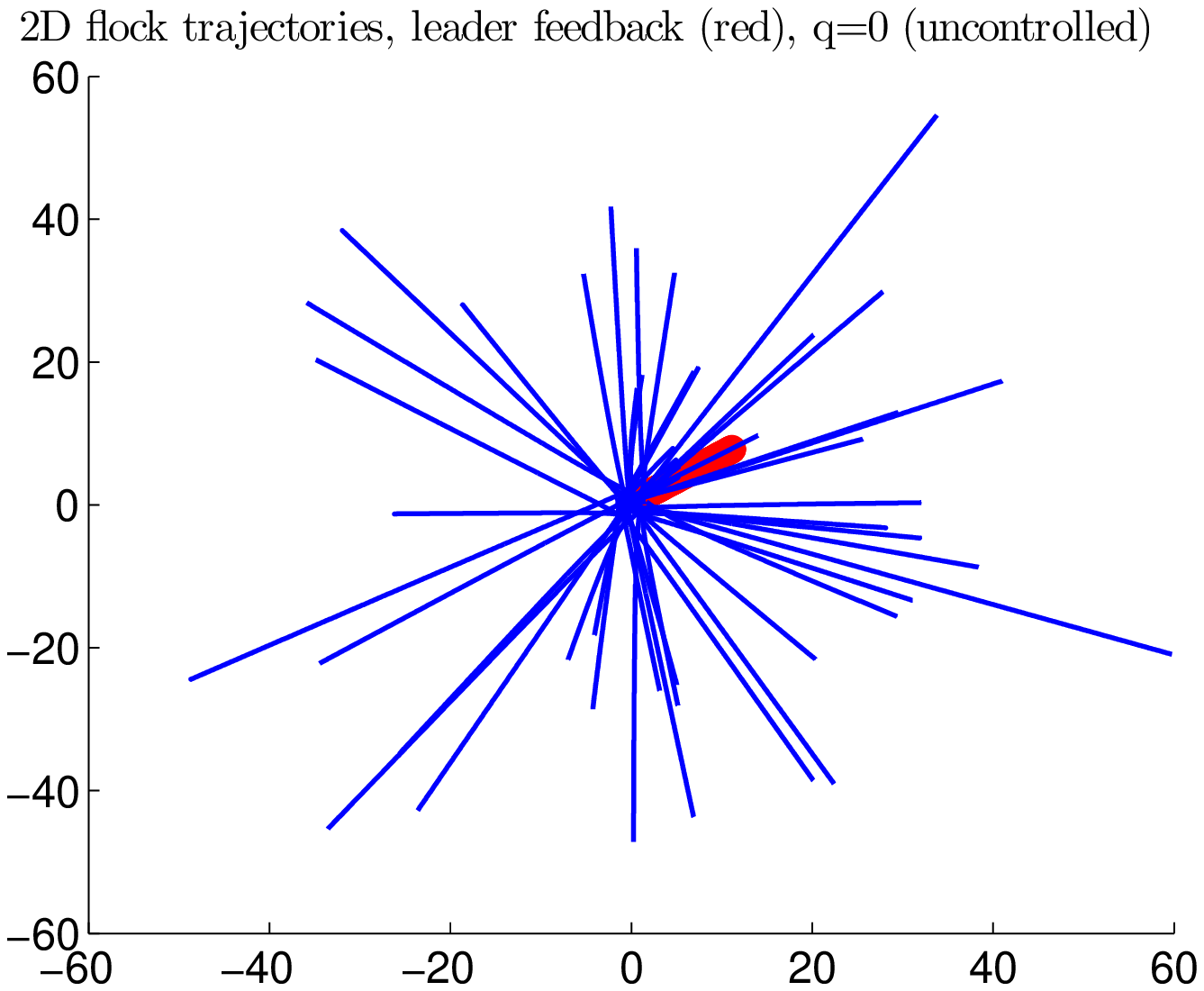,width=0.5\linewidth,clip=} &
\epsfig{file=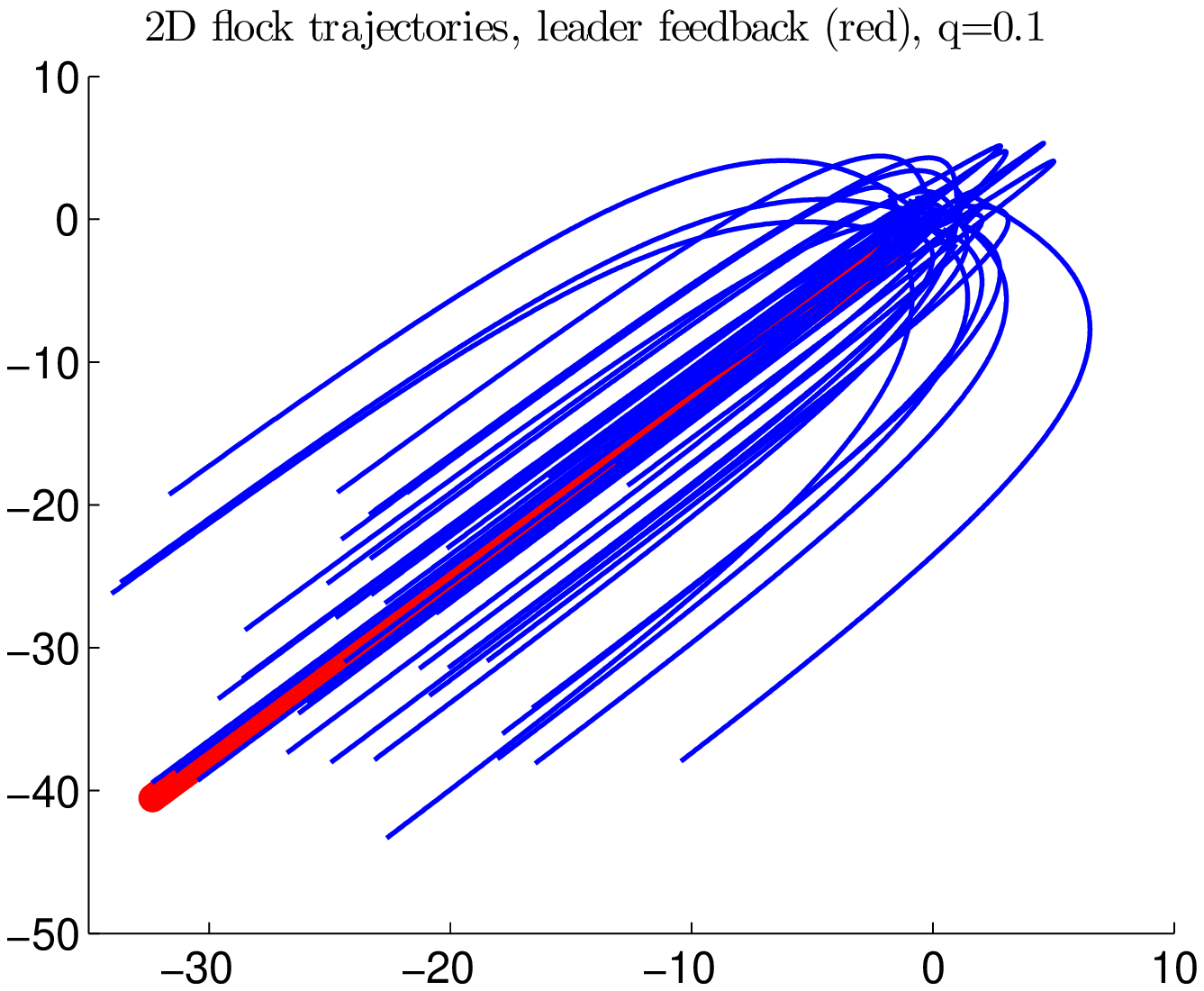,width=0.5\linewidth,clip=}\\
\epsfig{file=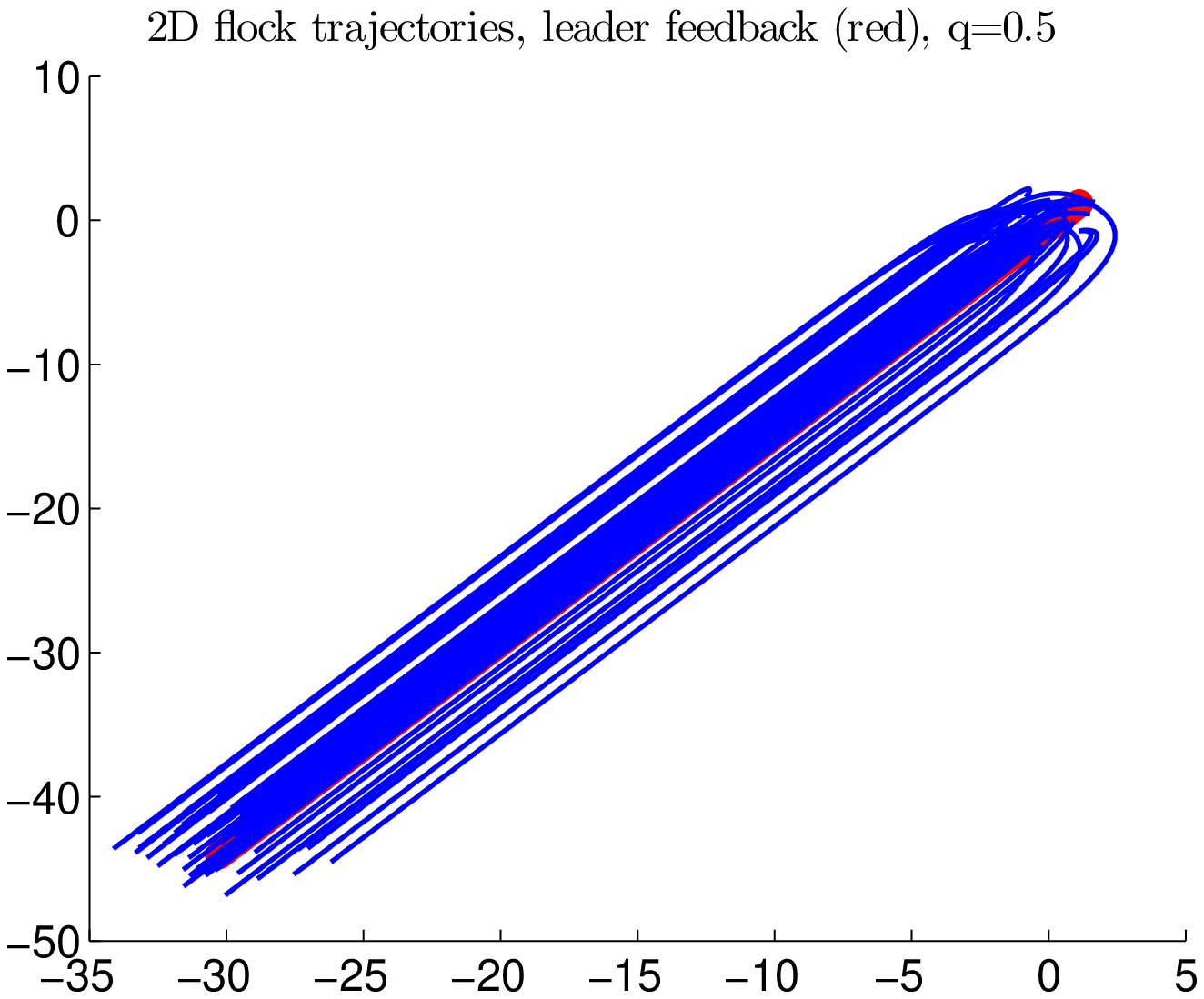,width=0.5\linewidth,clip=} &
\epsfig{file=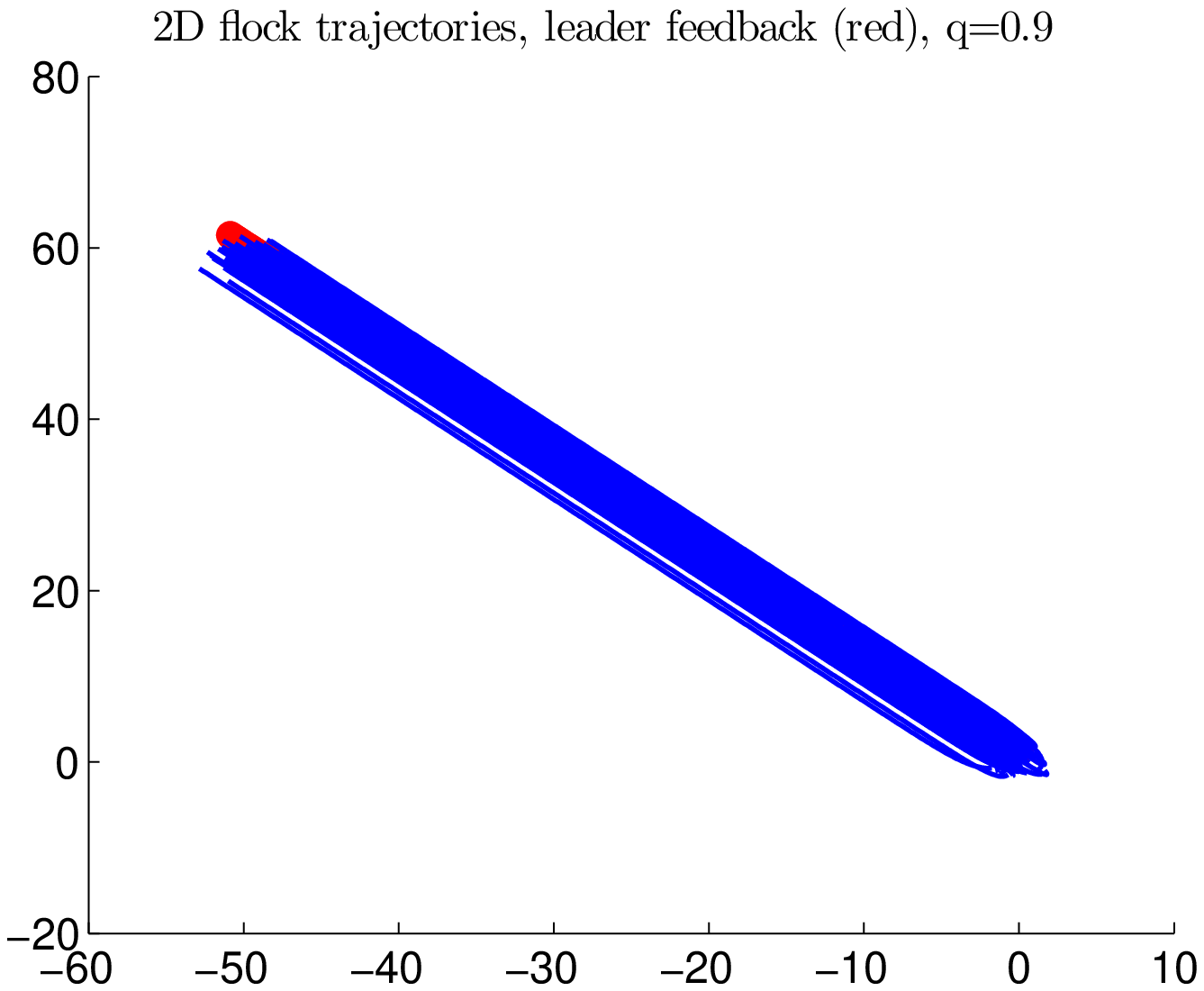,width=0.5\linewidth,clip=}
\end{tabular}}
\caption{Leader-based feedback control. Simulations with 100 agents, the value $q$ indicates the strength of the of the leader in the partial average. It can be observed how, as the strength of the leader is increased, convergent behavior is improved.}
\label{fig:1}
\end{figure}

\begin{figure}[!ht]
\centering
\resizebox{\textwidth}{!}{
\begin{tabular}{cc}
\epsfig{file=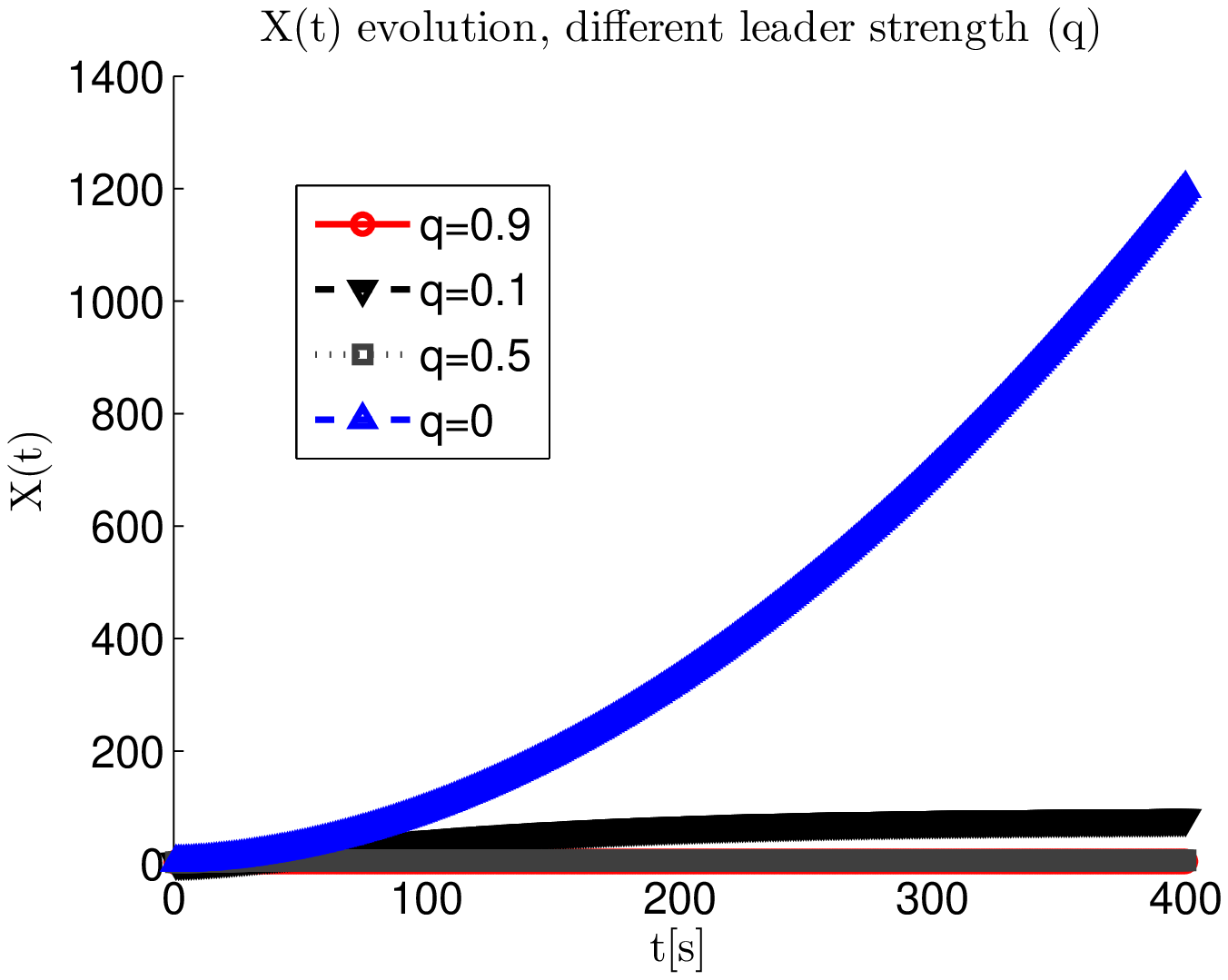,width=0.5\linewidth,clip=} &
\epsfig{file=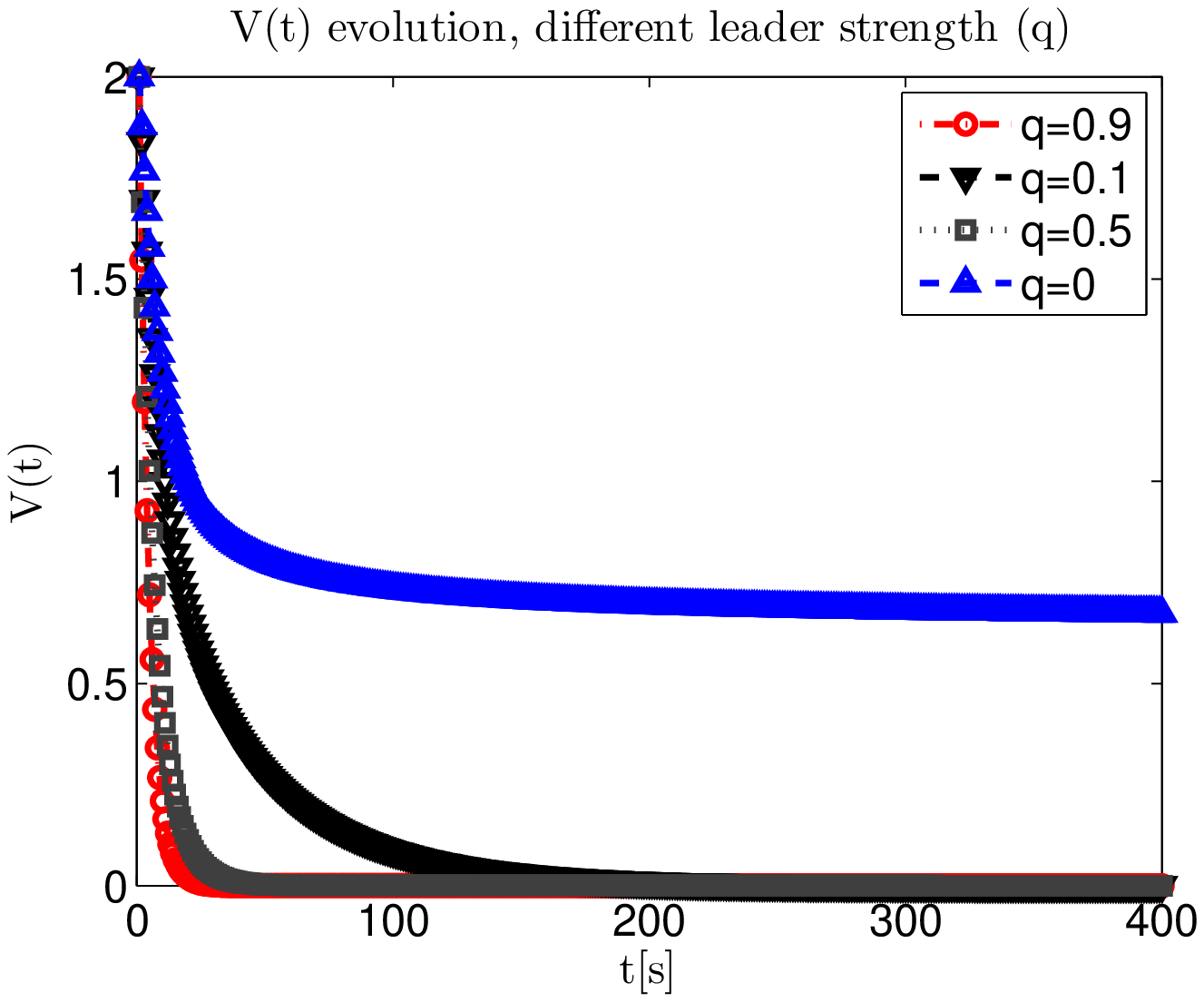,width=0.5\linewidth,clip=}
\end{tabular}}
\caption{Leader-based feedback control. Simulations with 100 agents, the value $q$ indicates the strength of the of the leader in the partial average. Evolution of $X(t)$ and $V(t)$ for the simulations in Figure \ref{fig:1}.}
\label{fig:12}
\end{figure}
\begin{figure}[!ht]
\centering
\resizebox{\textwidth}{!}{
\begin{tabular}{cc}
\epsfig{file=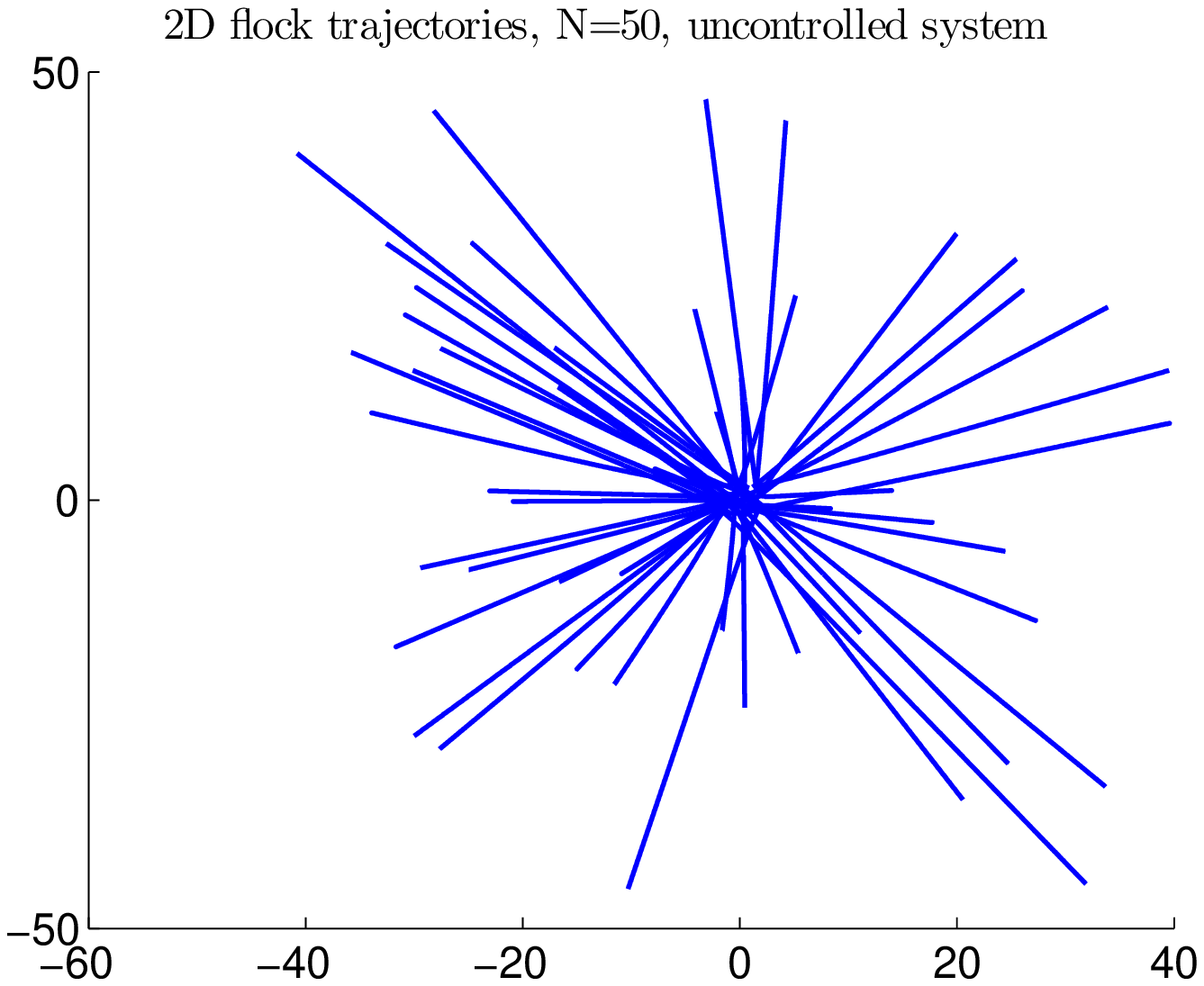,width=0.5\linewidth,clip=} &
\epsfig{file=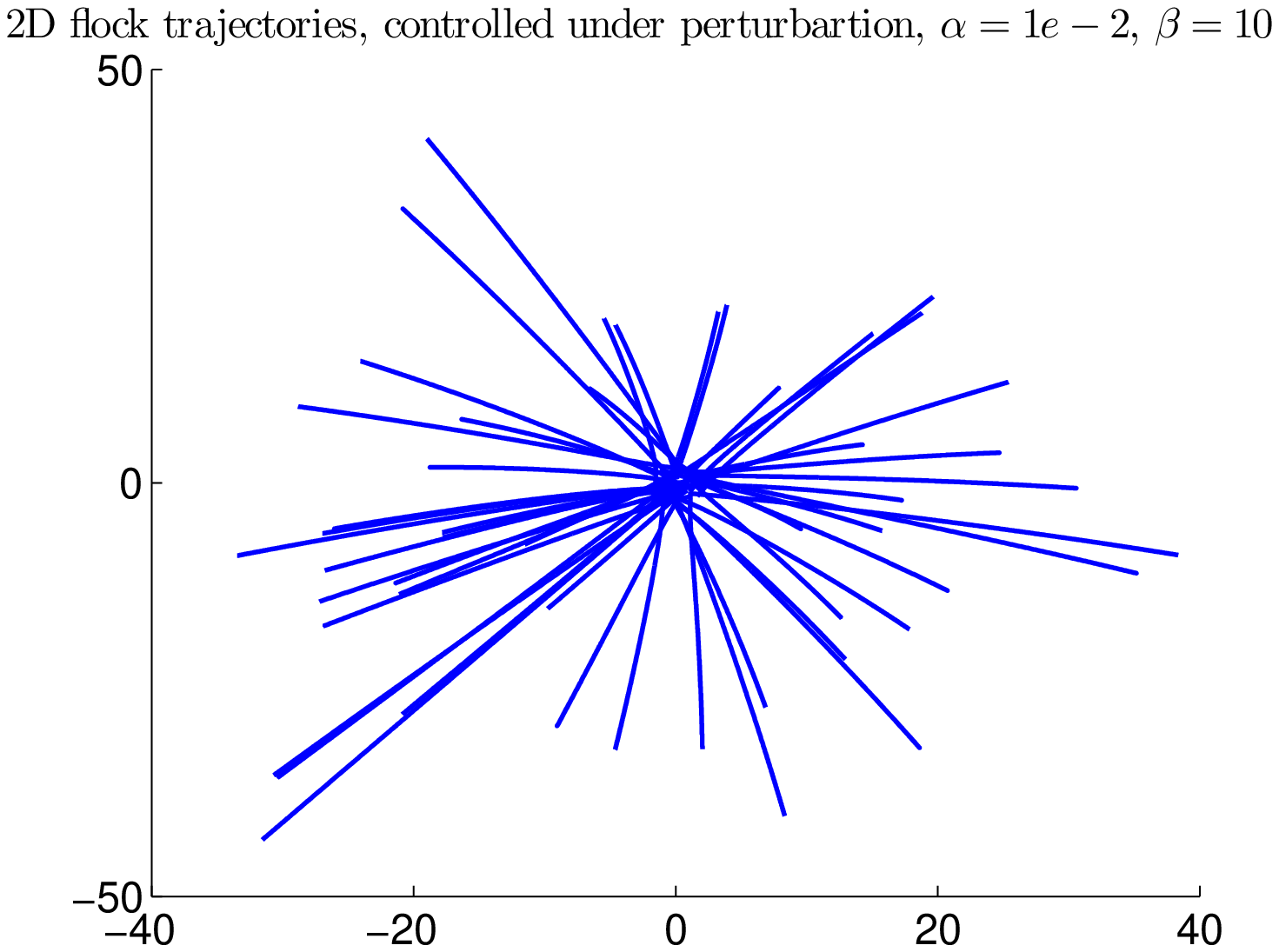,width=0.5\linewidth,clip=}\\
\epsfig{file=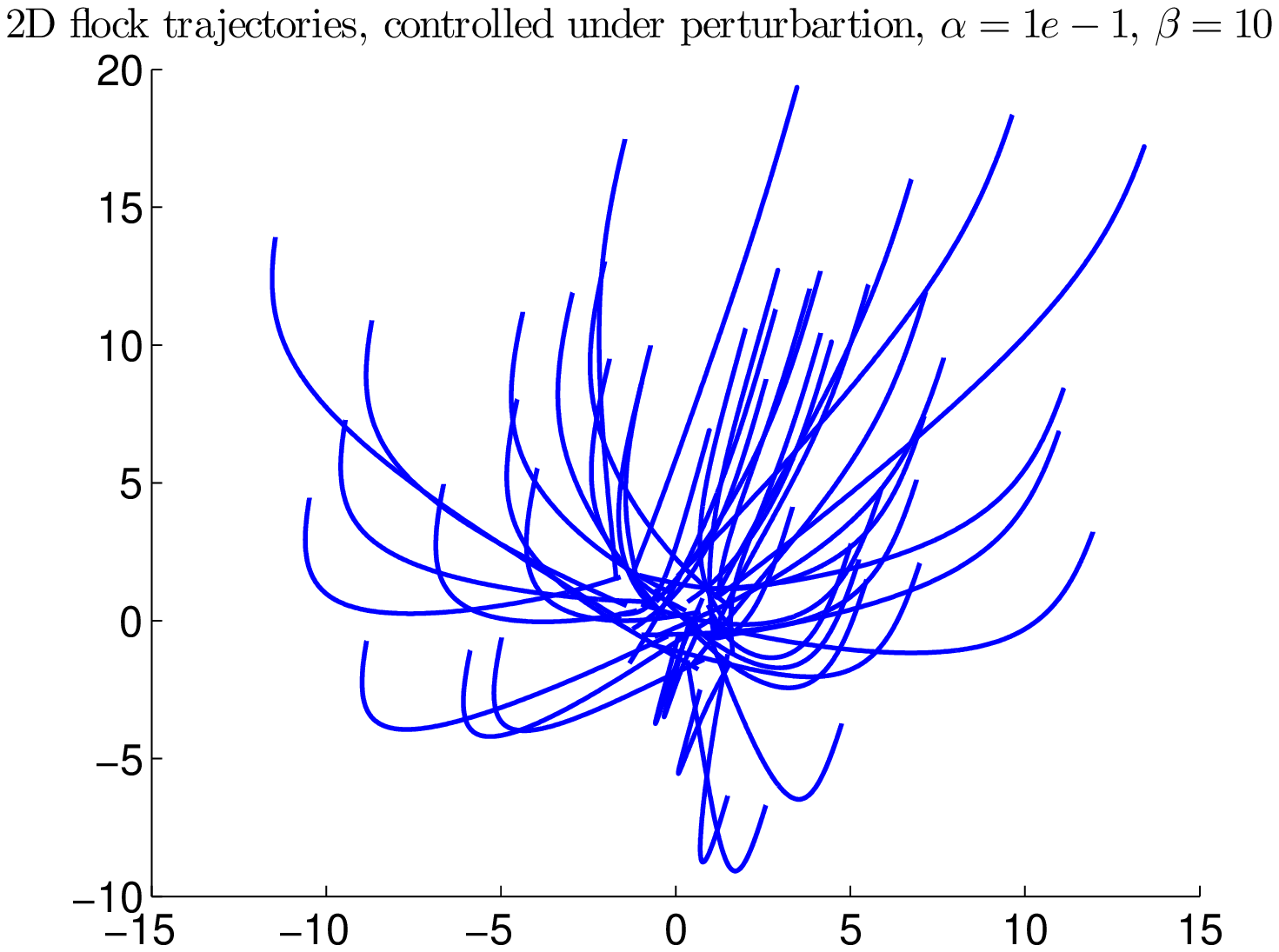,width=0.5\linewidth,clip=} &
\epsfig{file=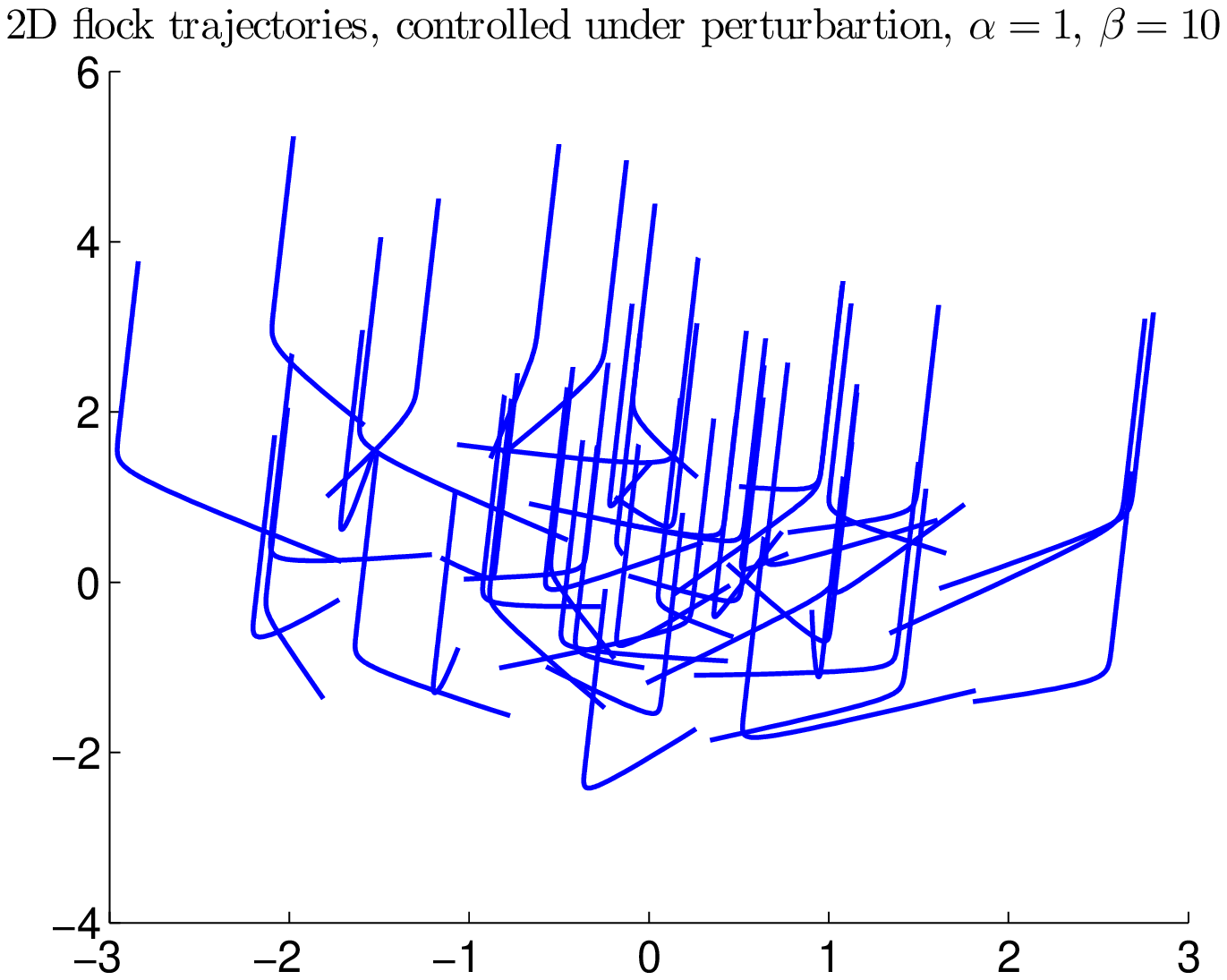,width=0.5\linewidth,clip=}
\end{tabular}}
\caption{Total feedback control under structured perturbations. For a fixed strong structured perturbation term ($\beta=10$), different energies for the unperturbed control term $\alpha$ generate different consensus behavior; the stronger the correct information term is, the faster consensus is achieved.}
\label{fig:2}
\end{figure}

\noindent\textbf{Feedback under perturbed information.} Next, we deal with the setting presented again in Section \ref{sec:cuckersmale_localmean}, by considering a feedback of the form
\[
\bm{u}_i=-\alpha(\bm{v}_i-\overline{\bm{v}})-\beta\Delta_i\,,\quad\overline{\bm{v}}=\frac{1}{N}\sum_{i=1}^N\bm{v}_i\,,
\]
where $\Delta_i=\Delta_i(t)$ represents a structured perturbation written as
\[
\Delta_i=\frac{1}{\eta_i(t)}\sum_{j=1}^N\omega_{ij}(\bm{v}_j-\overline{\bm{v}})\,.
\]
In particular, we address the case where the weighting function $\omega_{ij}$ corresponds to similar Cucker-Smale kernel as for the dynamics, i.e.,
\[\omega_{ij}=\frac{1}{(1+||\bm{x}_i-\bm{x}_j||^2)^{\epsilon}}\,,\quad \eta_i=\frac{1}{N}\sum_{j=1}^N \omega_{ij}\]
In this test, we fix a large value of $\beta=10$, representing a strong perturbation of the feedback, and a small value of $\epsilon=1e-5$, related to a disturbance which is distributed among all the agents. In Figures \ref{fig:2} and \ref{fig:22} , it is shown how increasing the value of $\alpha$, representing the energy of the \textsl{correct information feedback}, induces faster consensus emergence.

\begin{figure}[!ht]
\centering
\resizebox{\textwidth}{!}{
\begin{tabular}{cc}
\epsfig{file=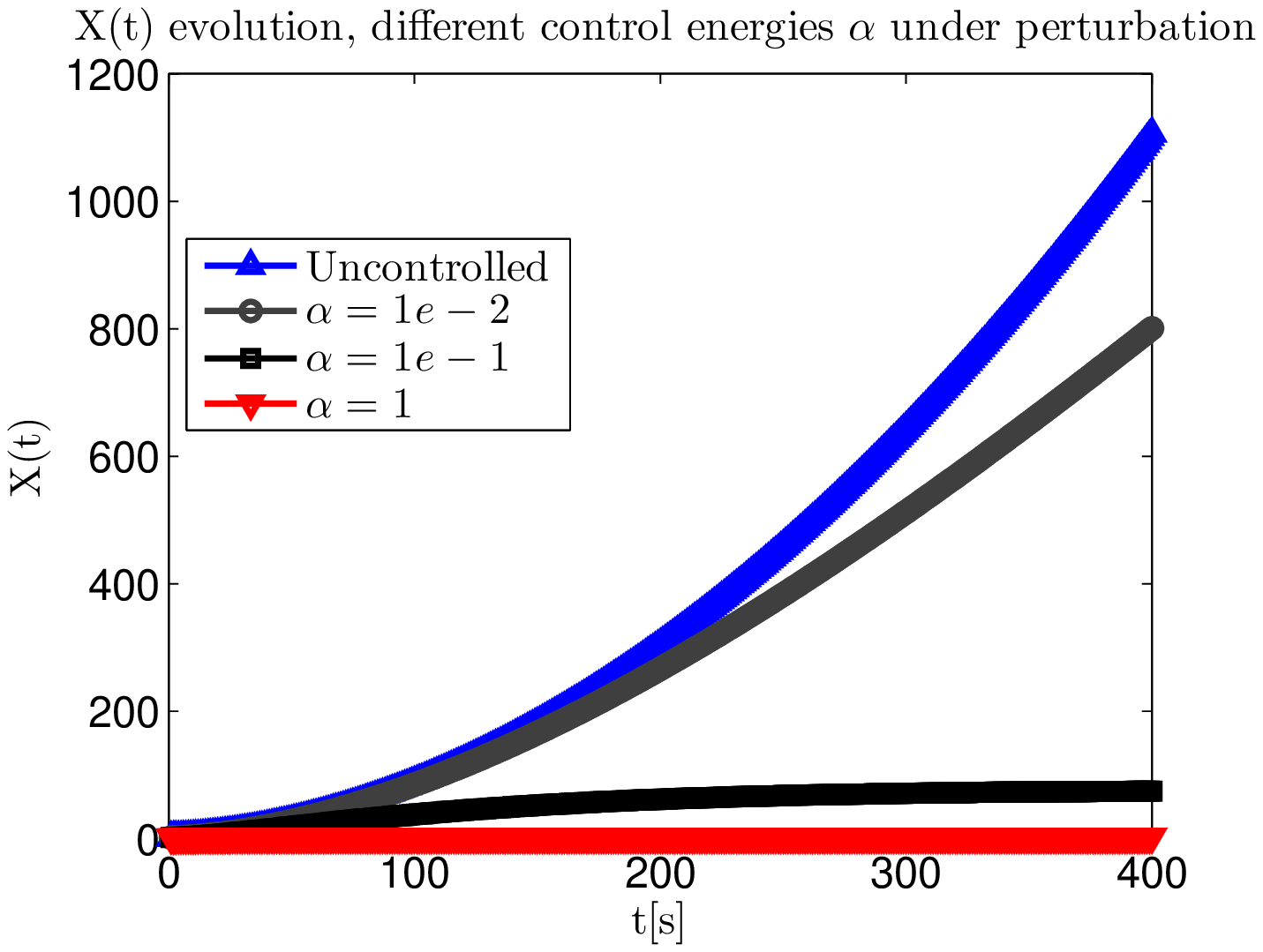,width=0.5\linewidth,clip=} &
\epsfig{file=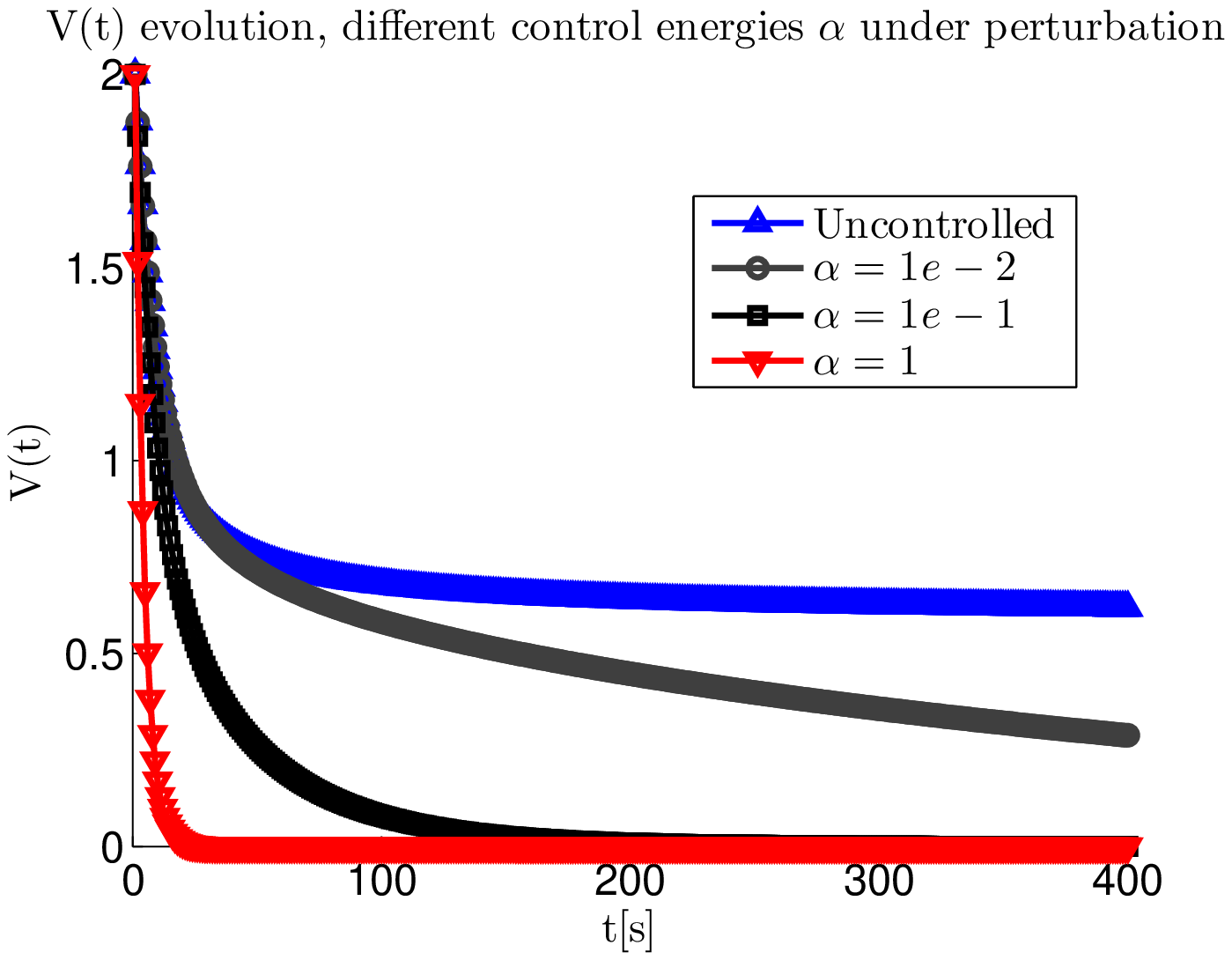,width=0.5\linewidth,clip=}
\end{tabular}}
\caption{Total feedback control under structured perturbations. Evolution of $X(t)$ and $V(t)$ for the simulations in Figure \ref{fig:2}.}
\label{fig:22}
\end{figure}

\noindent\textbf{Local feedback control.} The last test case studies the results presented in Section \ref{sec:localmeanR}, where the feedback is computed according to the local average in eq.\eqref{eq:truncated_perturbation}. Simulations in Figure \ref{fig:3} illustrate the setting. From an uncontrolled system, represented by a local feedback radius $R=0$, by increasing this quantity, partial flocking is consistently achieved, until  full consensus is observed for large radii mimicking a total information feedback control. From a theoretical perspective, this result is presented in Theorem \ref{th:HaHaKimExtended}, which describes a sufficient consensus region for feedback control based on local averages. This theorem recovers on its asymptotics previous results in \cite{HaHaKim}, and \cite{CFPT}, related to consensus regions for uncontrolled and fully controlled systems under total information feedback, respectively.
It has been reported in the literature \cite{cafotove10}, that estimates for consensus regions such as the one provided by Theorem \ref{thm:hhk}, are not sharp in many situations. In this direction, we proceed to contrast the theoretical consensus estimates with the numerical evidence. For this purpose, for a fixed number of agents, we span a large set of possible initial configurations determined by different values of $(X,V)$. For every pair $(X,V)$ we randomly generate a set of 20 initial conditions, and we simulate for a sufficiently large time frame. We measure consensus according to a threshold established on the final value of $V$; we consider that consensus has been achieved if the final value of $V$ is lower or equal to $1e-5$. We proceed by computing empirical probabilities of consensus for every point of our state space $(X,V)$;  results in this direction are presented in Figures \ref{fig:4} and \ref{fig:5}. We first consider the simplified case of 2 agents; according to \cite{CFPT}, for this particular case, the consensus region estimate provided by Theorem \ref{thm:hhk} is sharp, as illustrated is by the results presented in Figure \ref{fig:4}. Furthermore, it is also the case for Theorem \ref{th:HaHaKimExtended}; for $R>0$, the consensus region predicted by the theorem coincides with the numerically observed ones.
\begin{figure}[!ht]
\centering
\resizebox{\textwidth}{!}{
\begin{tabular}{cc}
\epsfig{file=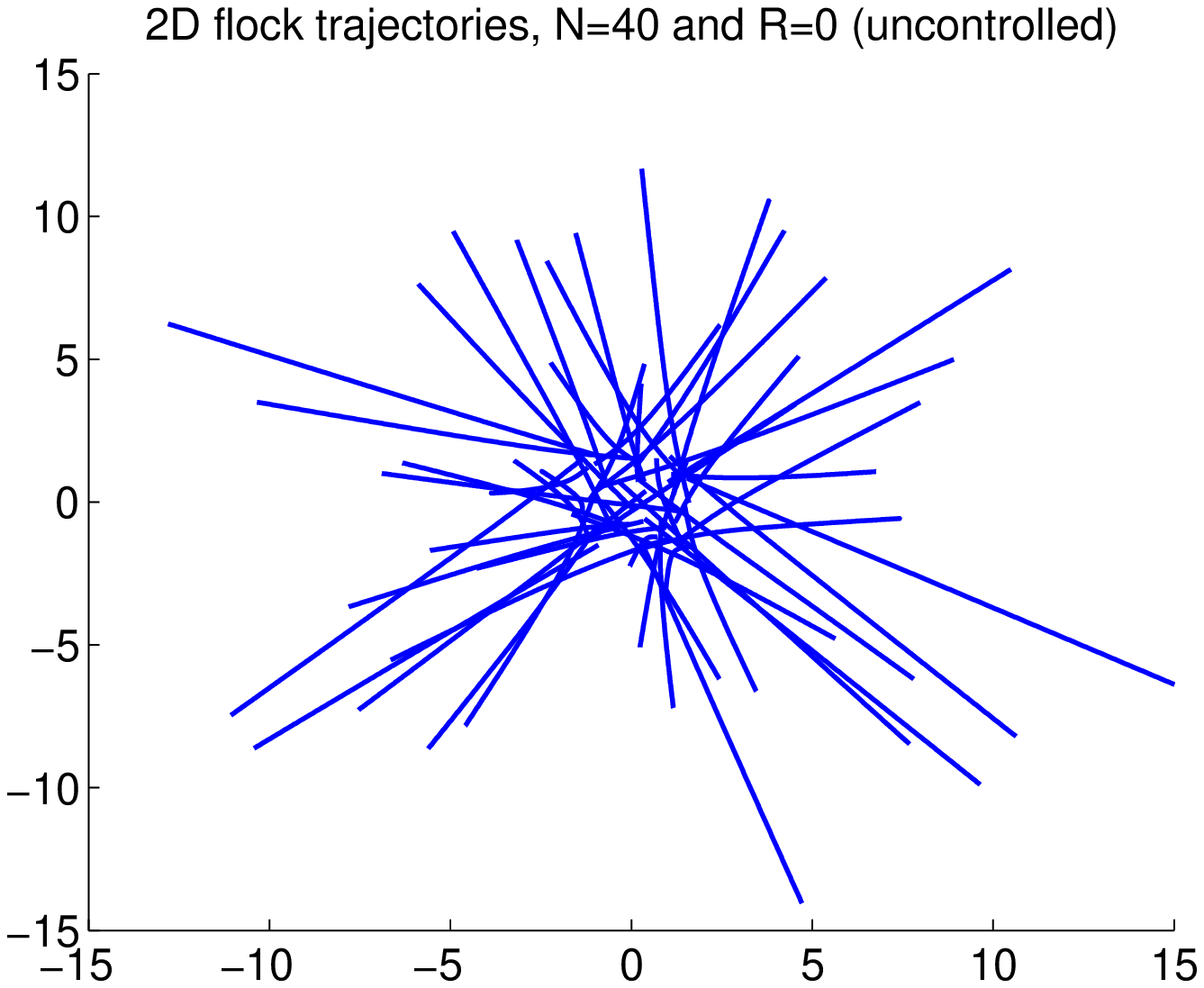,width=0.5\linewidth,clip=} &
\epsfig{file=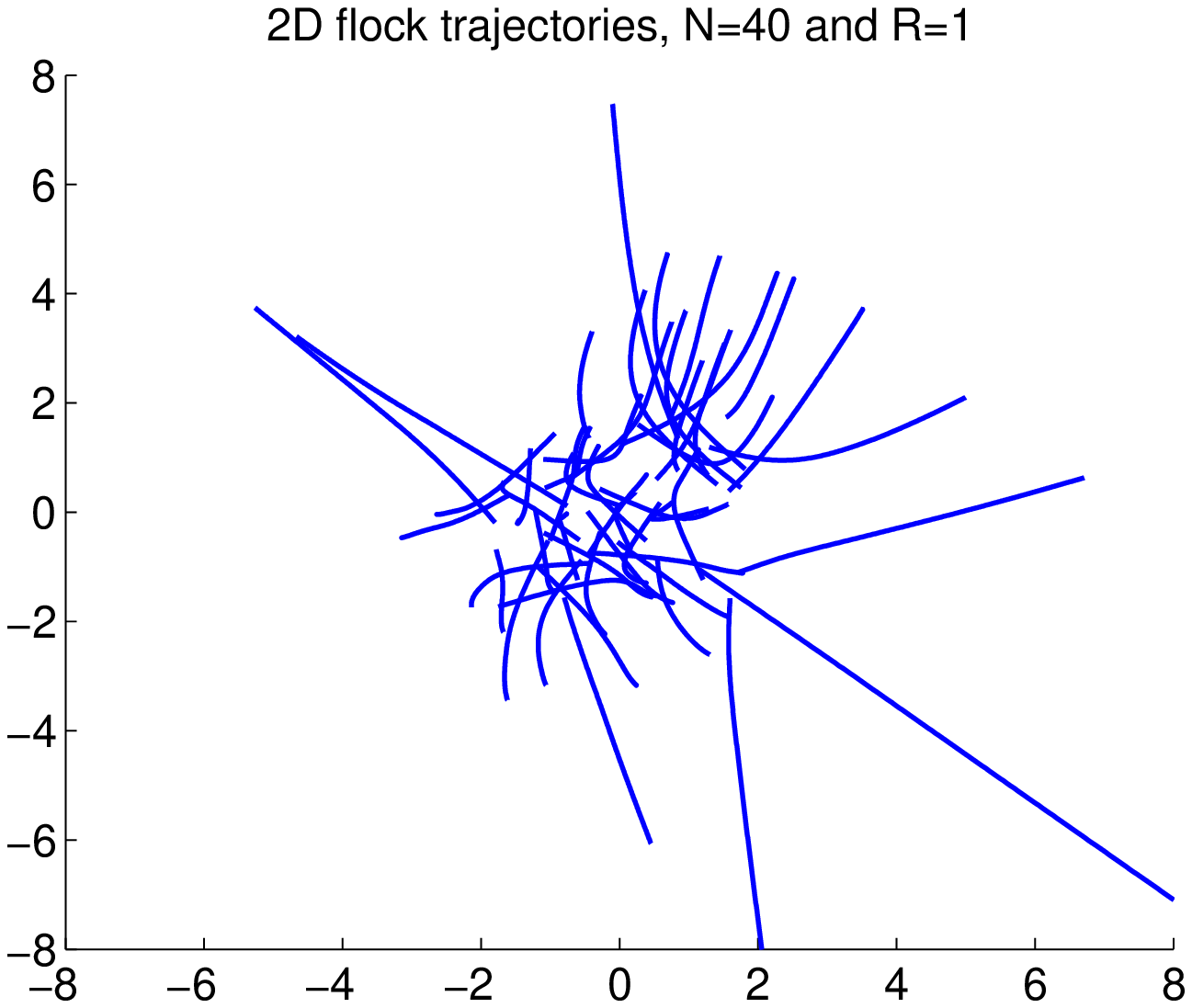,width=0.5\linewidth,clip=}\\
\epsfig{file=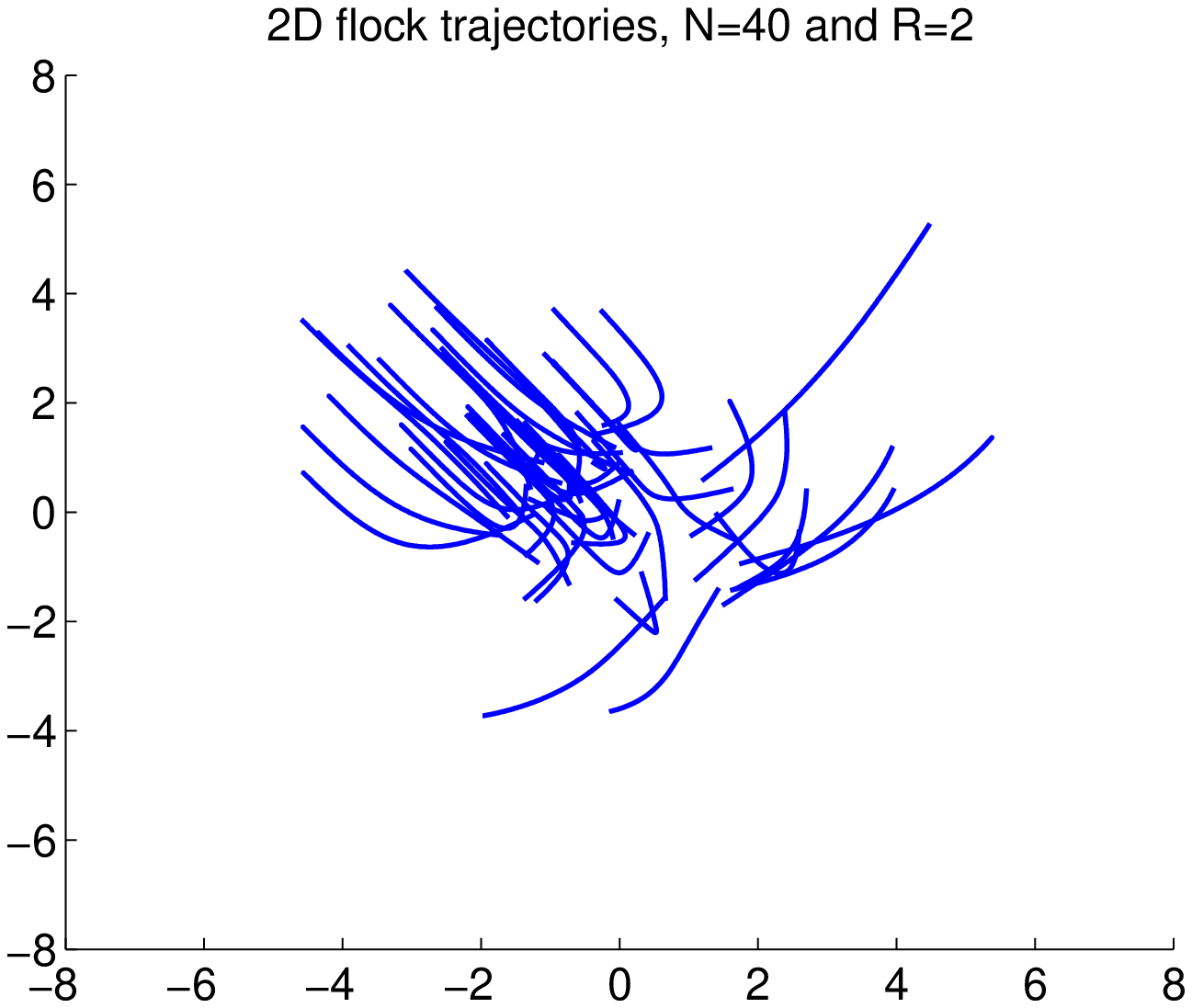,width=0.5\linewidth,clip=} &
\epsfig{file=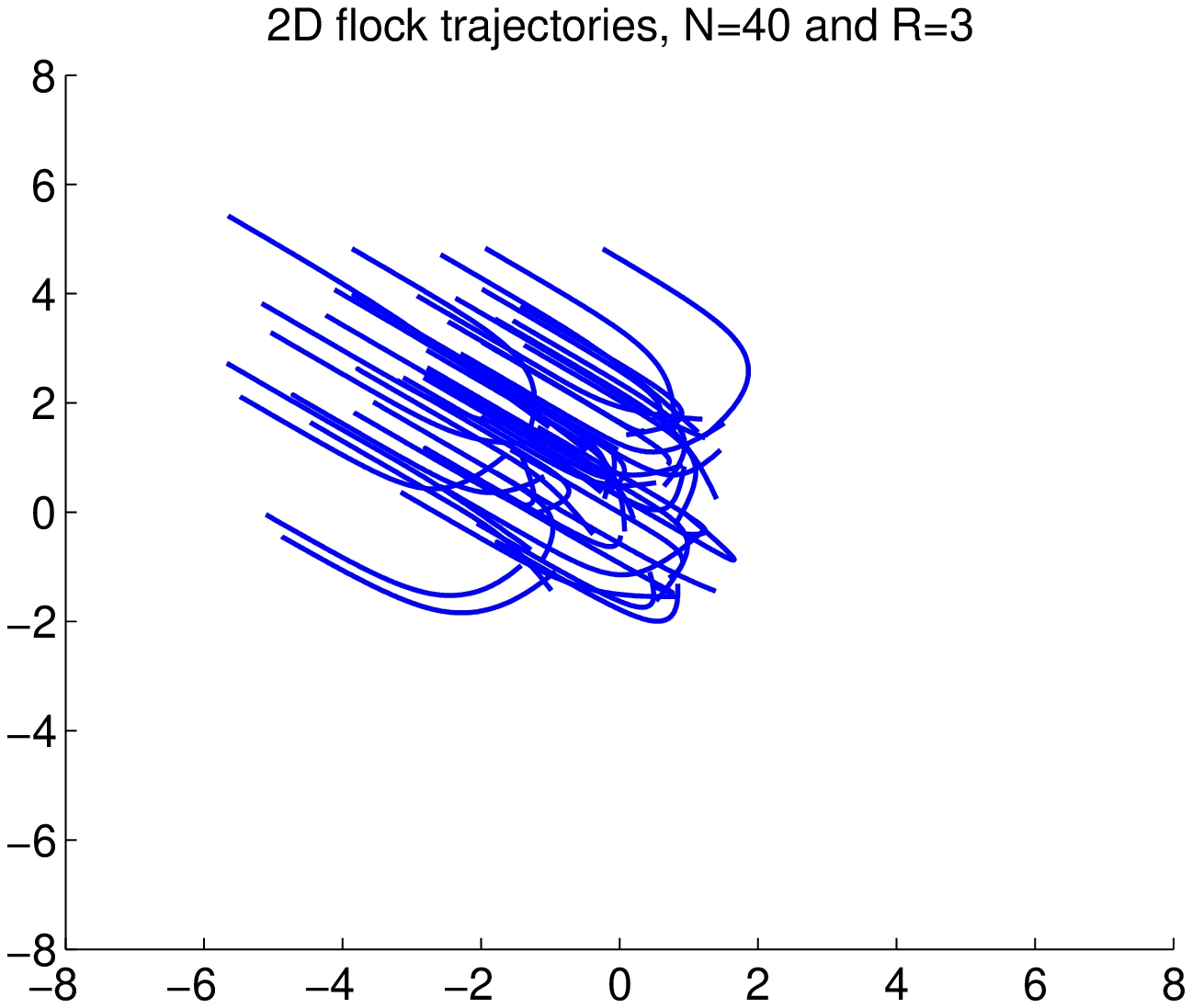,width=0.5\linewidth,clip=}\\
\epsfig{file=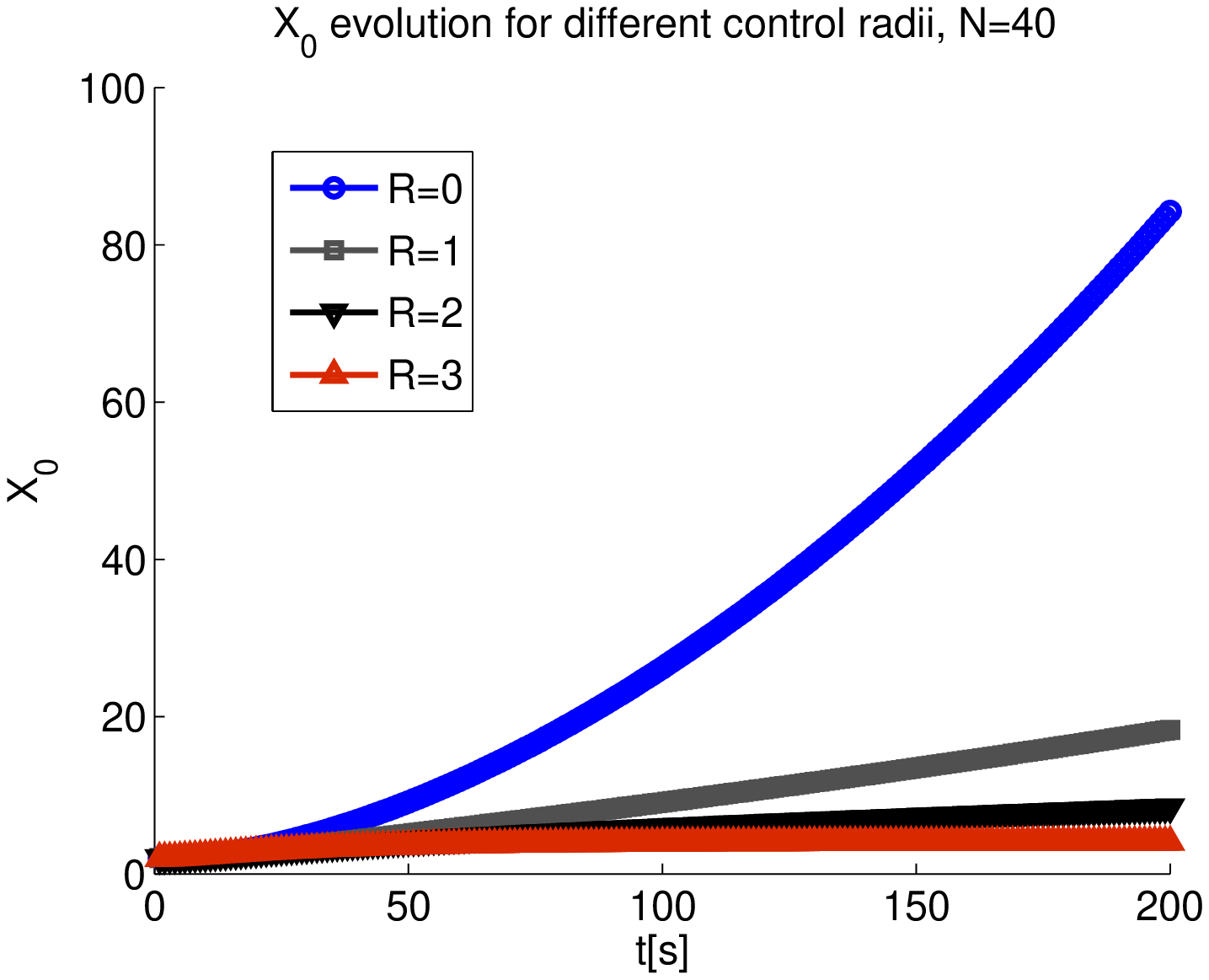,width=0.5\linewidth,clip=} &
\epsfig{file=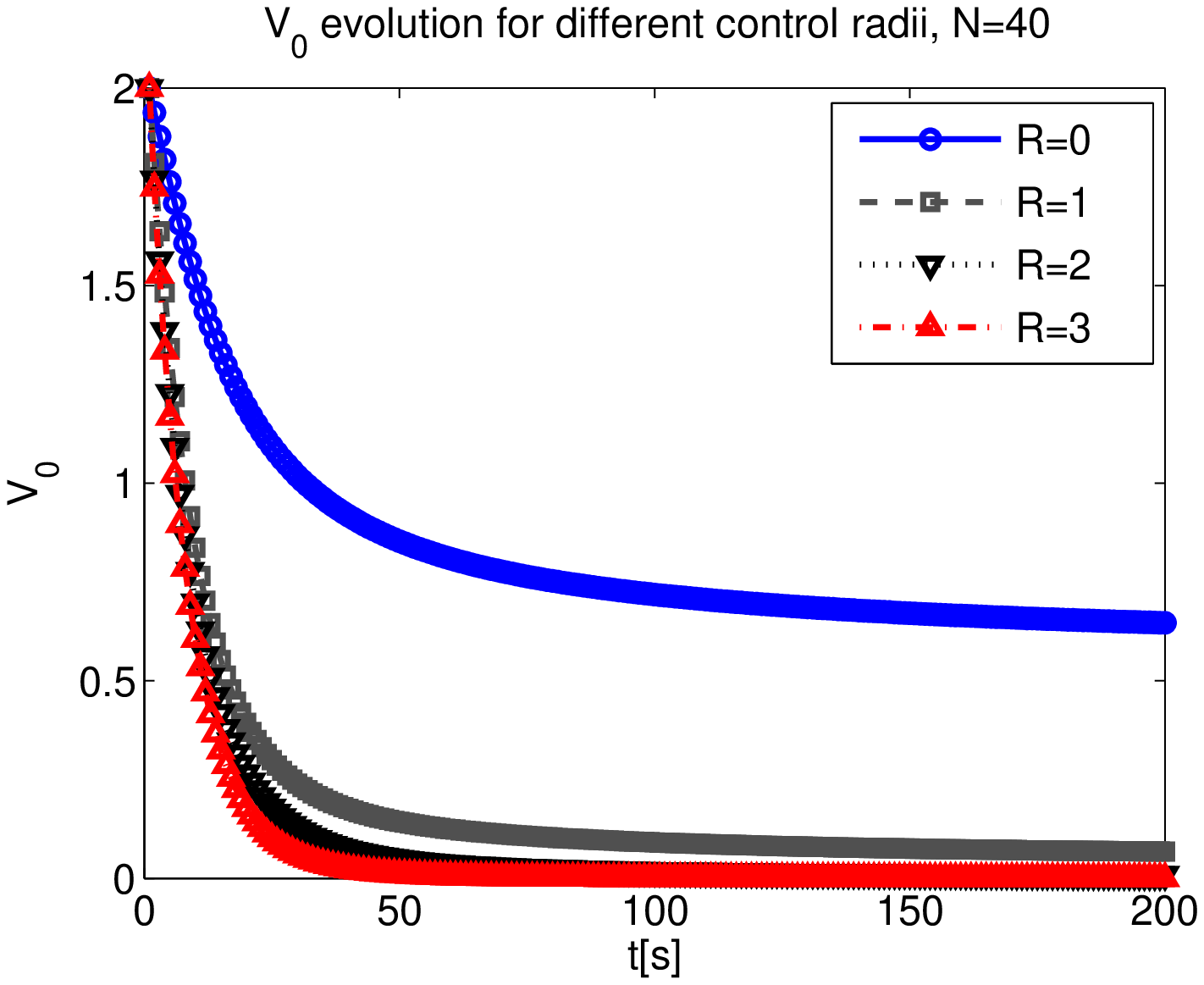,width=0.5\linewidth,clip=}
\end{tabular}}
\caption{Local feedback control. Simulations with $N=40$ agents, and  different control radii $R$. By increasing the value of $R$ the systems transits from uncontrolled behavior, to partial flocking, up to total, fast flocking.}
\label{fig:3}
\end{figure}

\begin{figure}
\centering
\resizebox{\textwidth}{!}{
\begin{tabular}{cc}
\epsfig{file=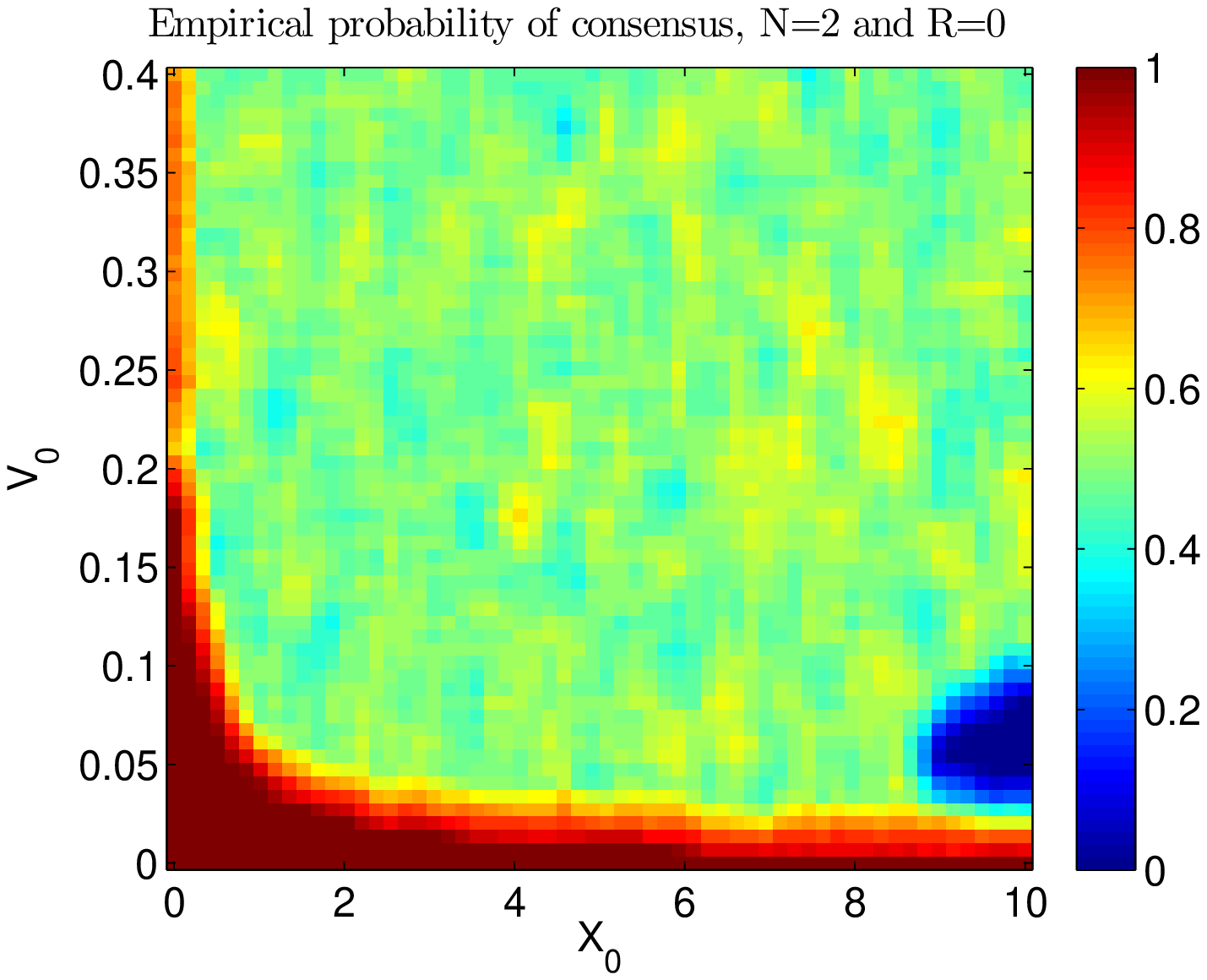,width=0.5\linewidth,clip=} &
\epsfig{file=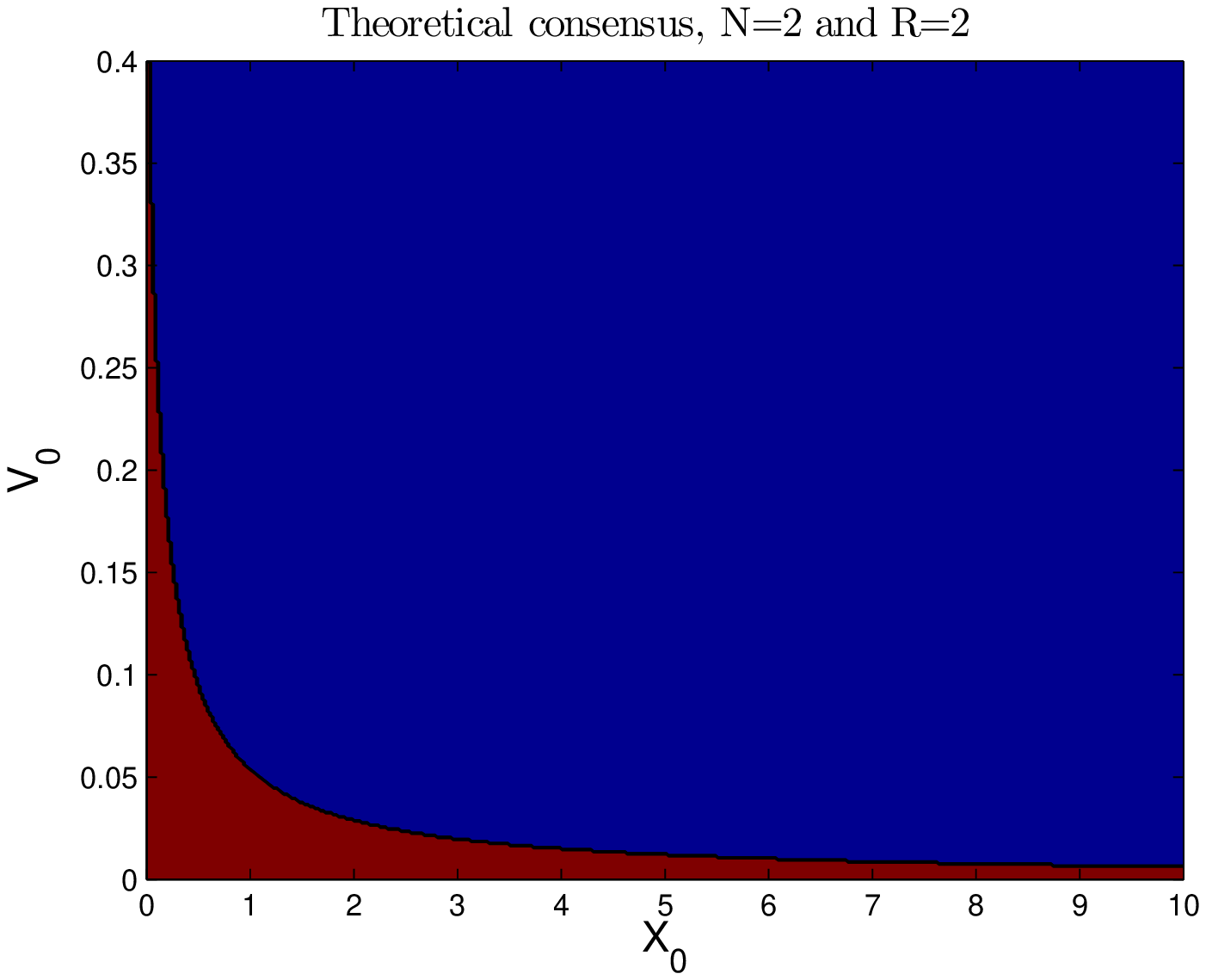,width=0.5\linewidth,clip=}\\
\epsfig{file=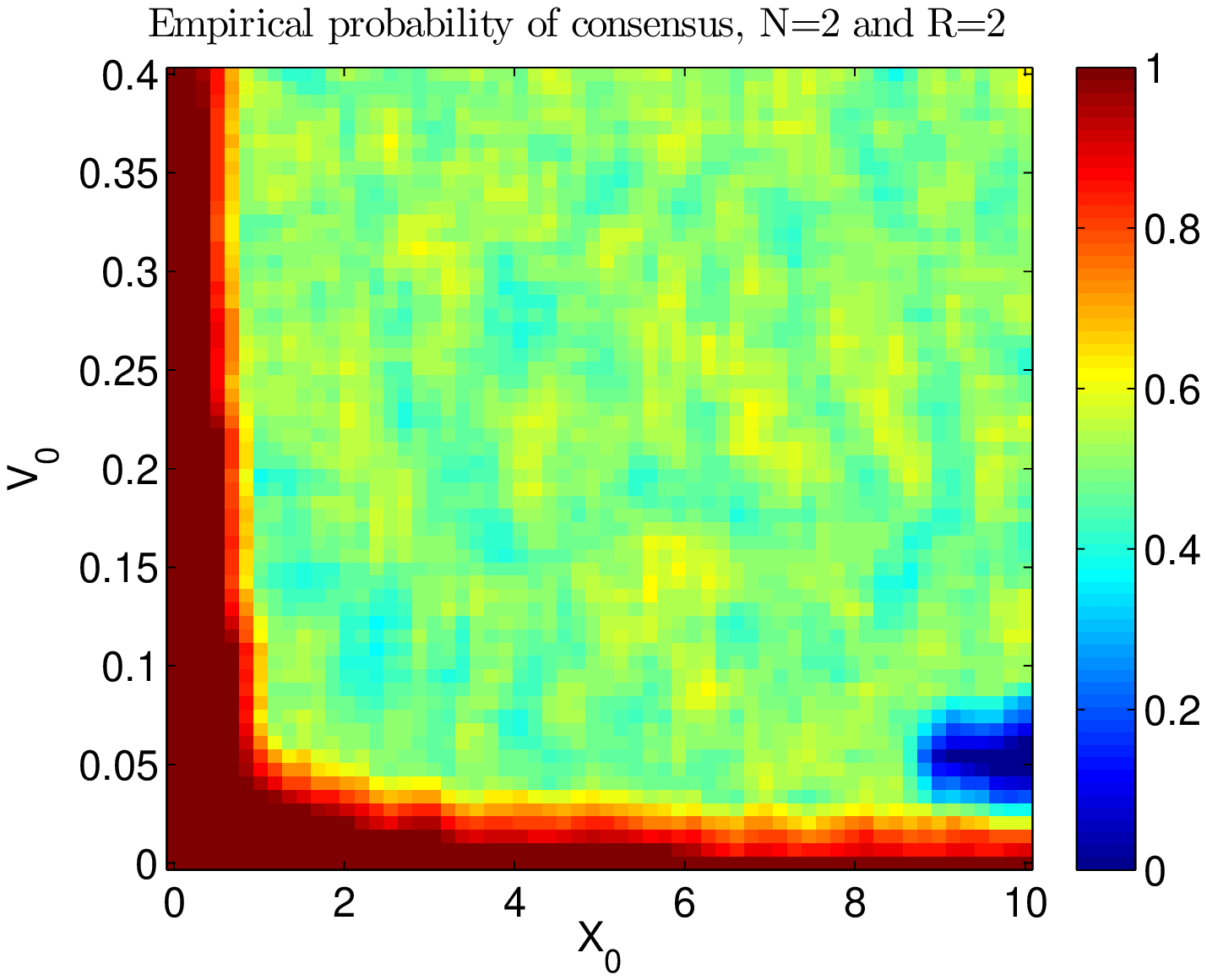,width=0.5\linewidth,clip=} &
\epsfig{file=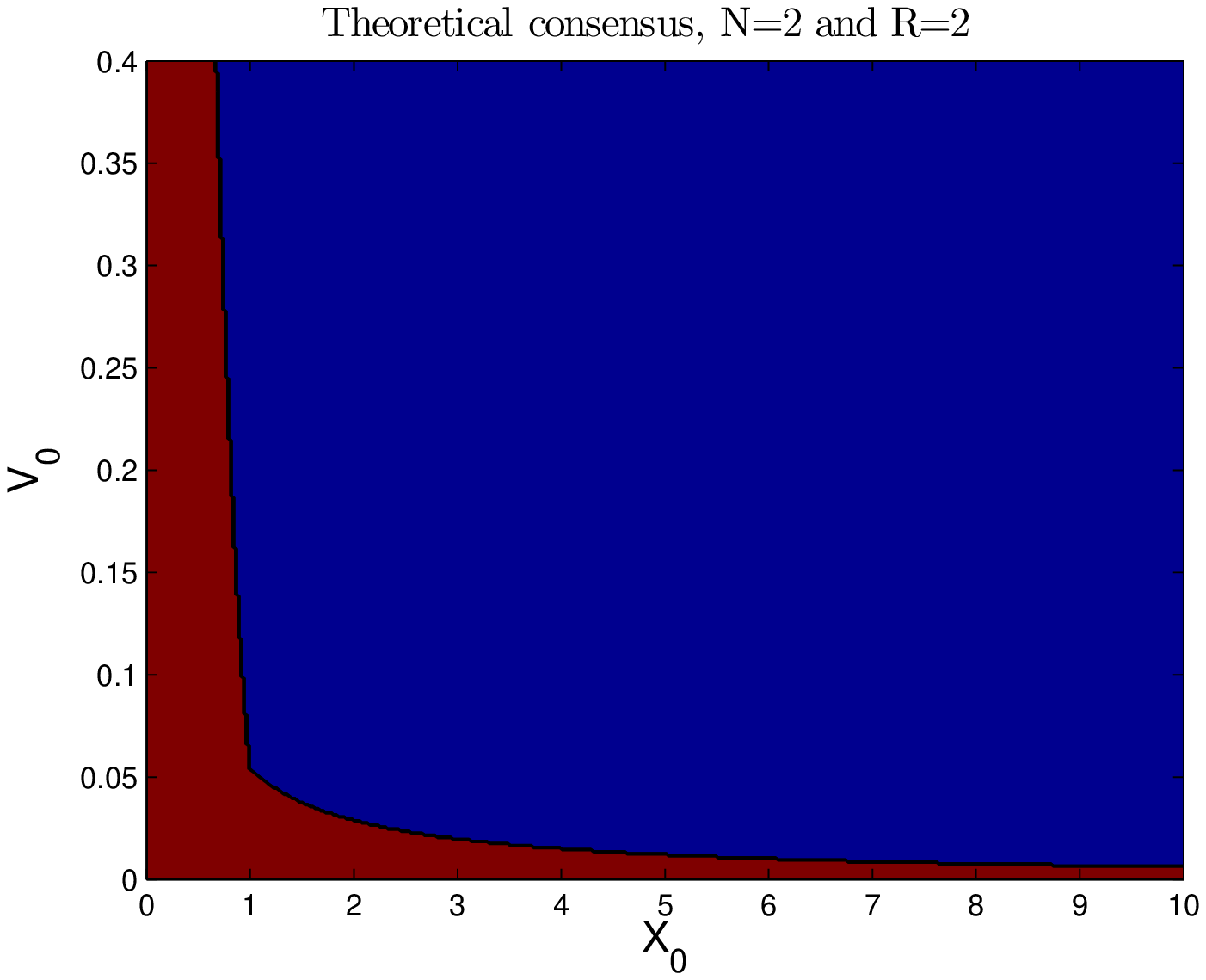,width=0.5\linewidth,clip=}
\end{tabular}}
\caption{Local feedback control. Empirical consensus regions and theoretical estimates for two-agent systems.}
\label{fig:4}
\end{figure}
Figure \ref{fig:5} illustrates the case when a larger number of agents is considered. In a similar way as for Theorem \ref{thm:hhk}, the consensus region estimate is conservative if compared with the region where numerical experiments exhibit convergent behavior. Nevertheless, Theorem \ref{th:HaHaKimExtended} is consistent in the sense that the theoretical consensus region increases gradually as $R$ grows, eventually covering any initial configuration, which is the case of the total information feedback control, as presented in \cite[Proposition 2]{CFPT}. The numerical experiments also confirm this phenomena, as shown in Figure \ref{fig:6} , where contour lines showing the $80\%$ probability of consensus for different radii locate farther from the origin as $R$ increases.
\begin{figure}
\centering
\resizebox{\textwidth}{!}{
\begin{tabular}{cc}
\epsfig{file=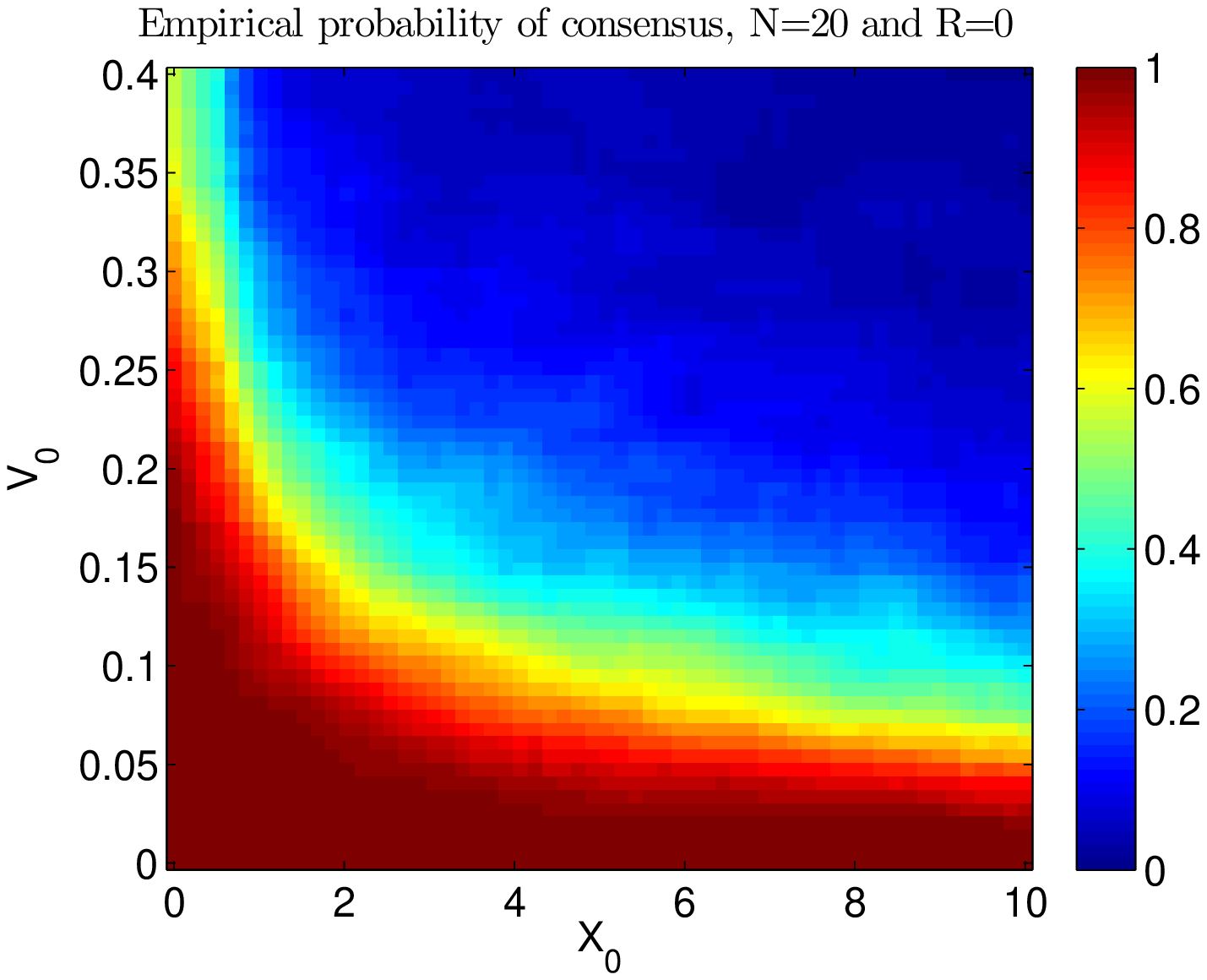,width=0.5\linewidth,clip=} &
\epsfig{file=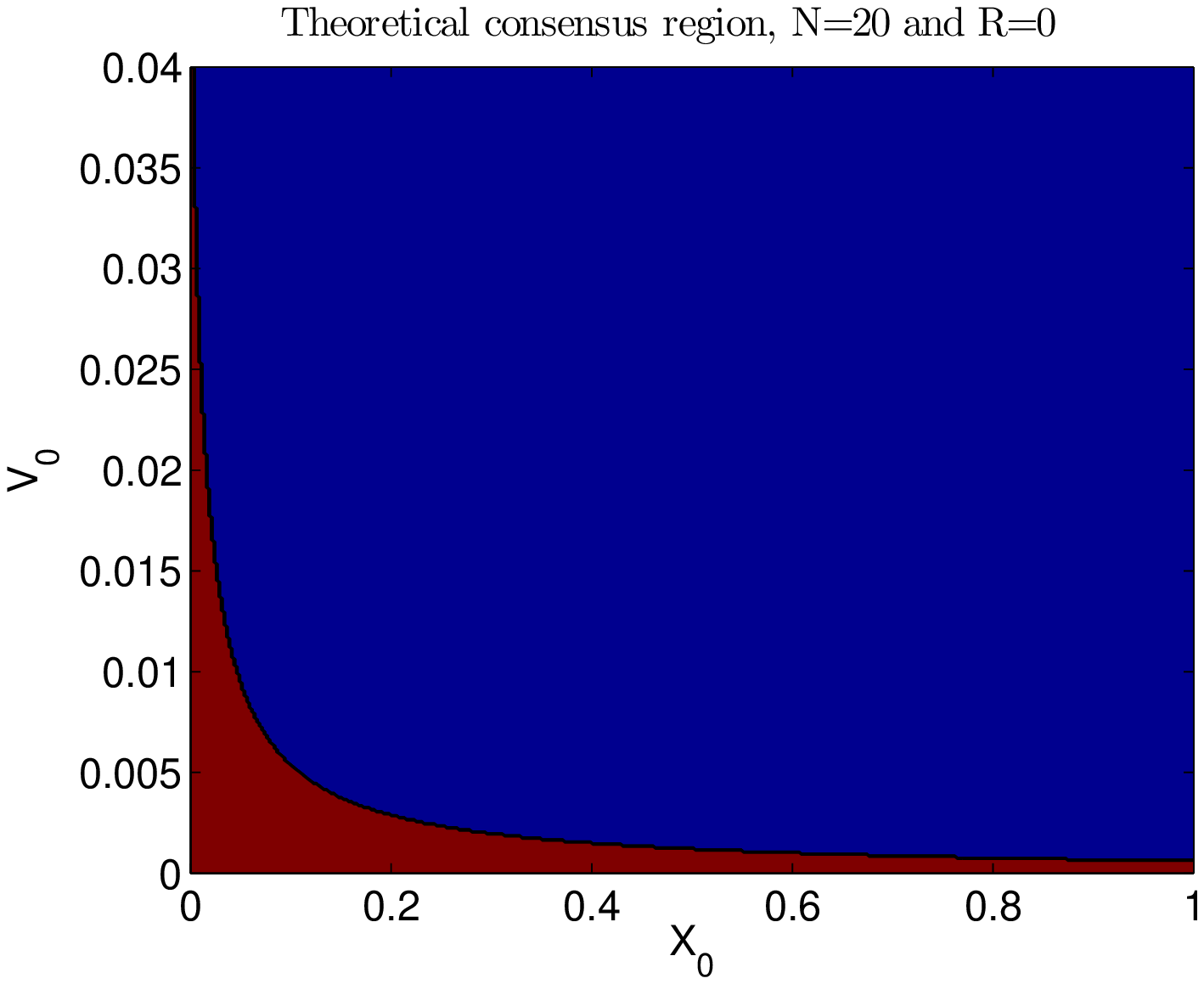,width=0.5\linewidth,clip=} \\
\epsfig{file=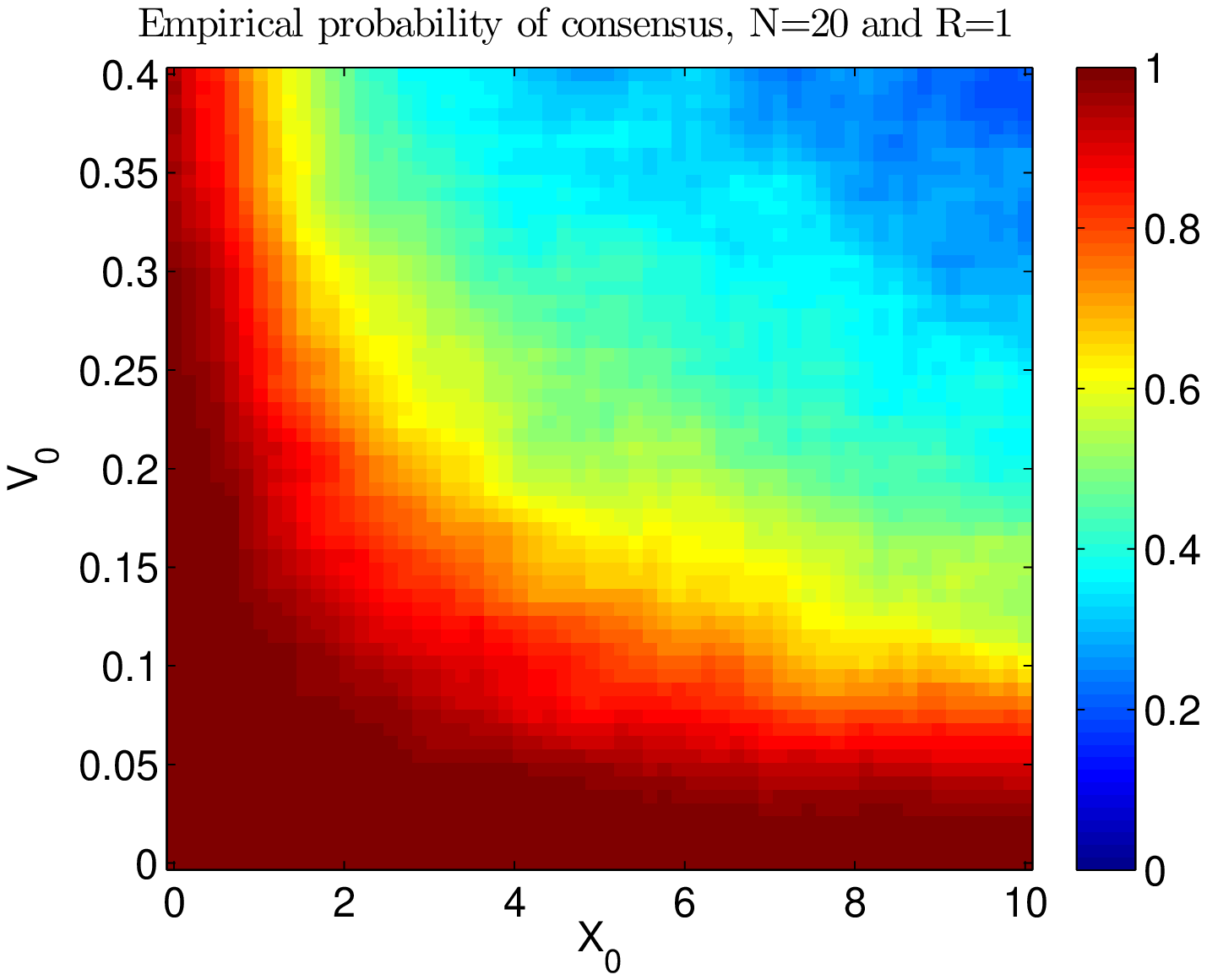,width=0.5\linewidth,clip=}&
\epsfig{file=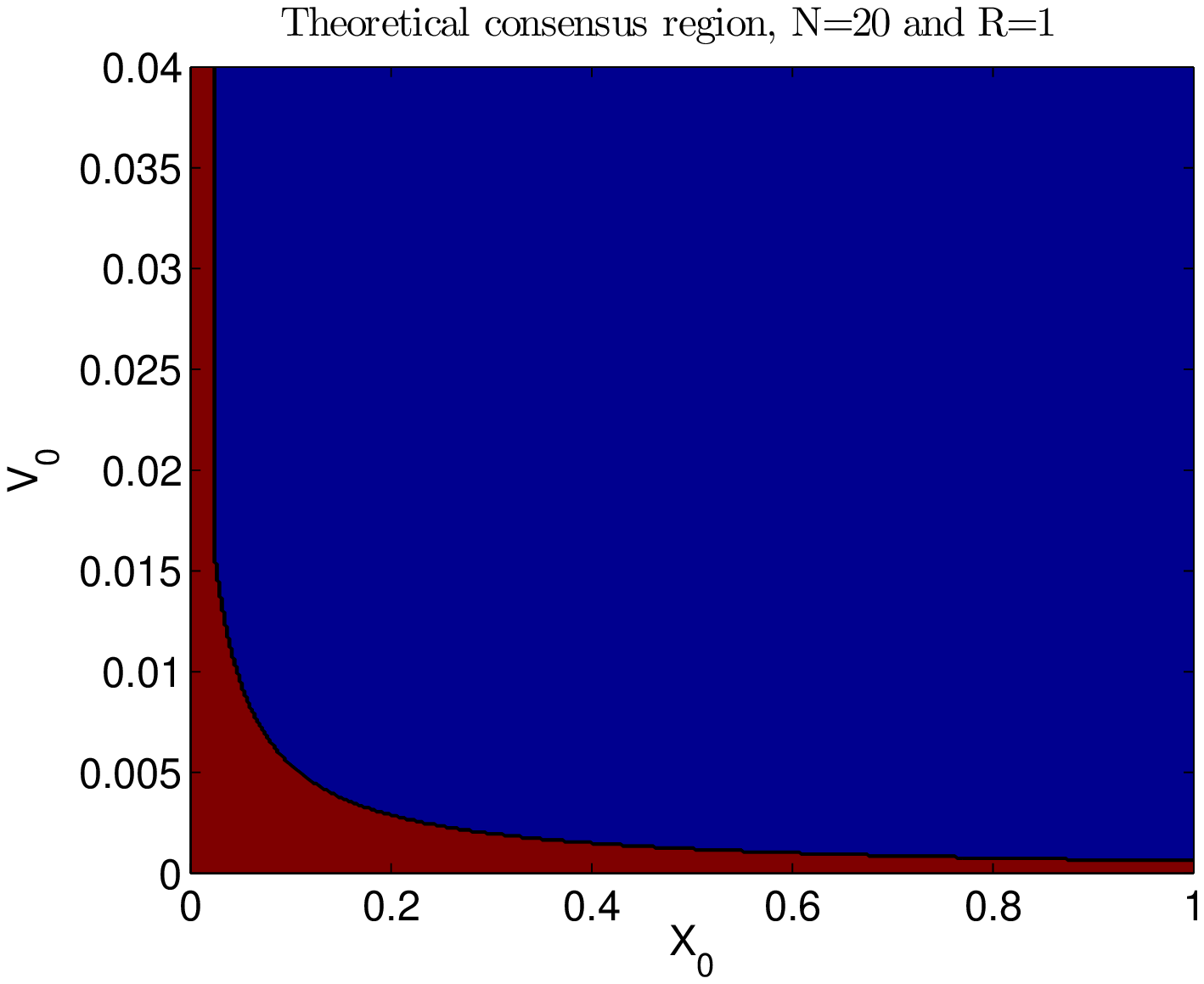,width=0.5\linewidth,clip=} \\
\epsfig{file=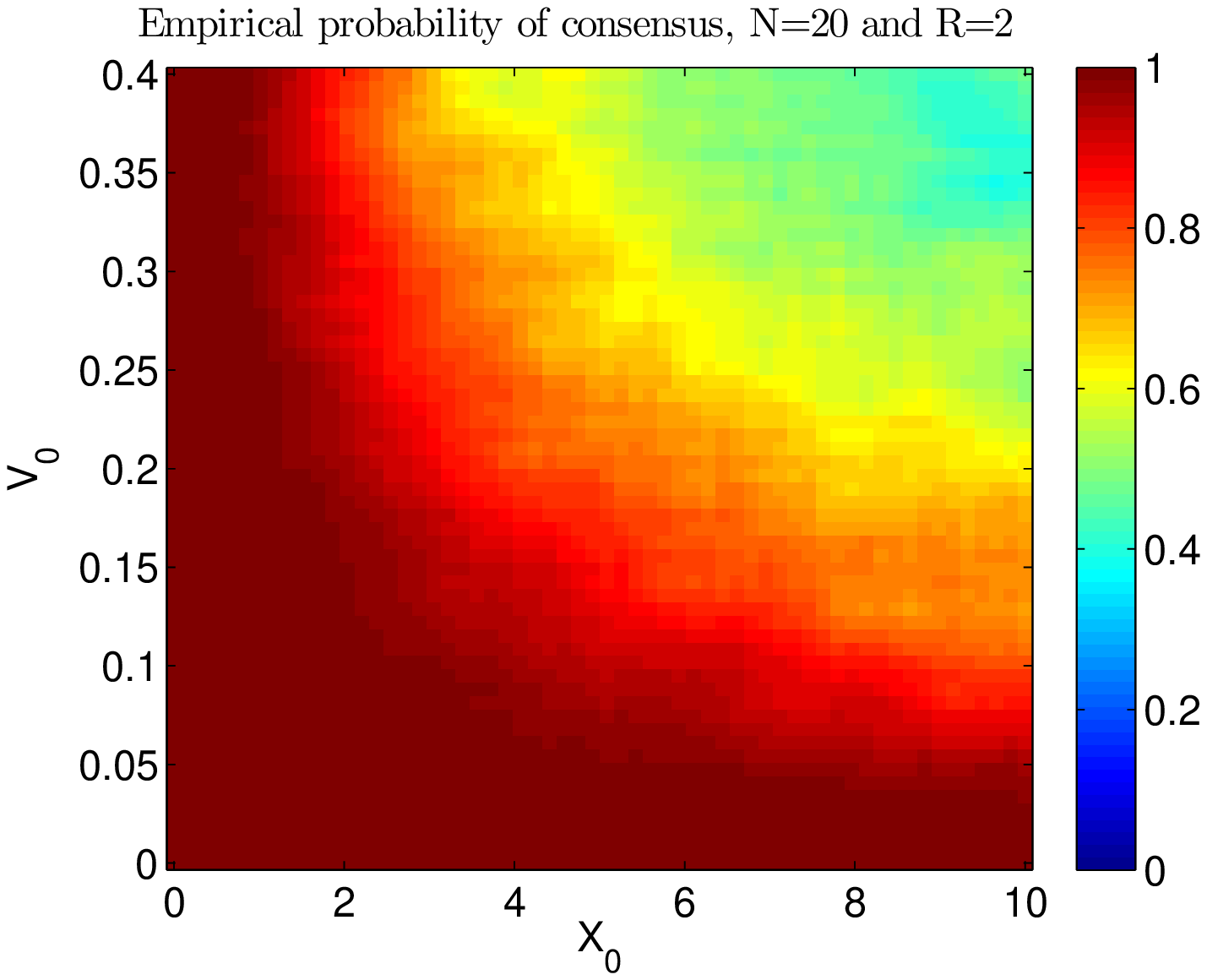,width=0.5\linewidth,clip=} &
\epsfig{file=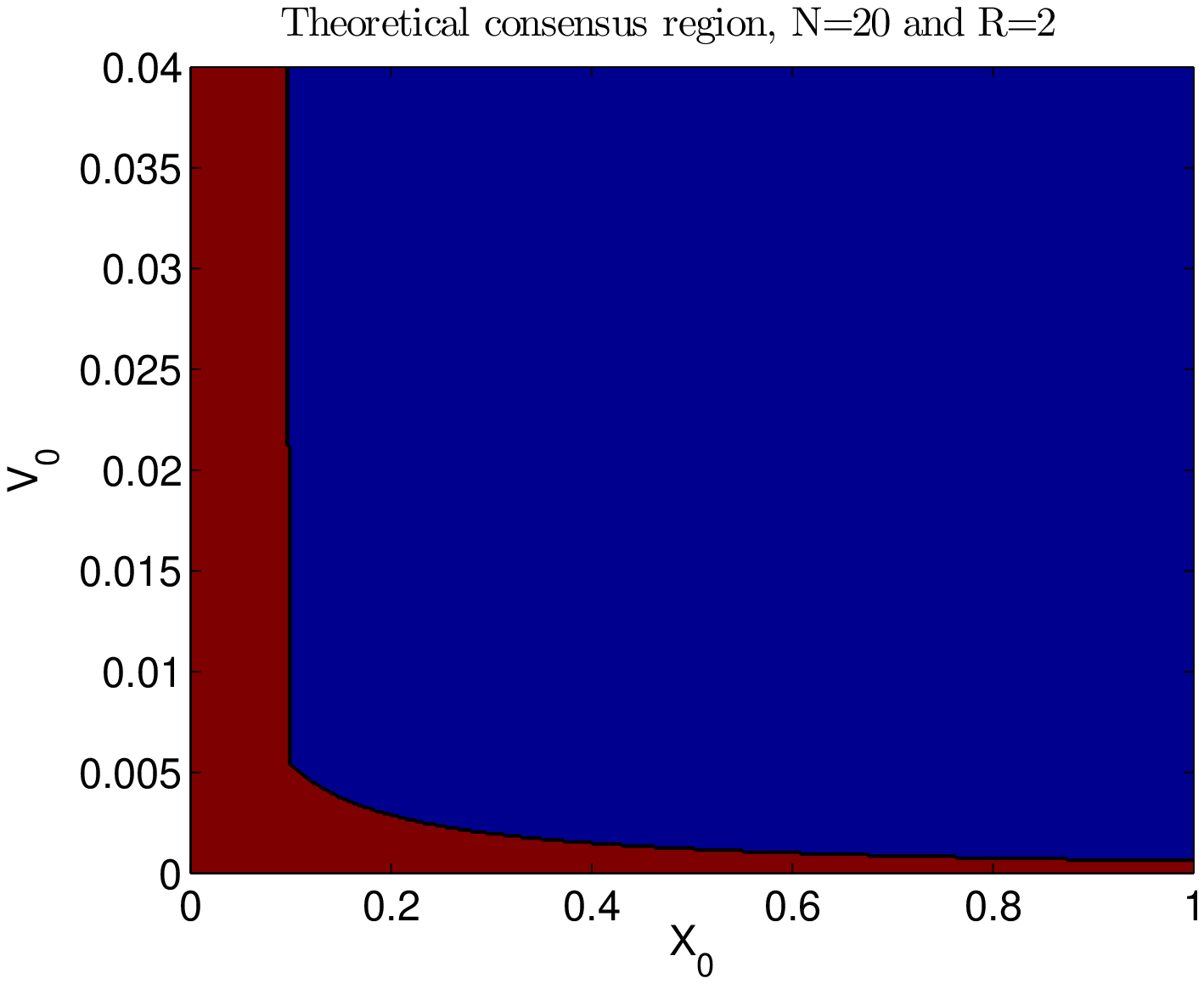,width=0.5\linewidth,clip=}
\end{tabular}}
\caption{Local feedback control. Empirical consensus regions and theoretical estimates for $N=20$ agents and different control radii $R$.}
\label{fig:5}
\end{figure}

\begin{figure}
\centering
\epsfig{file=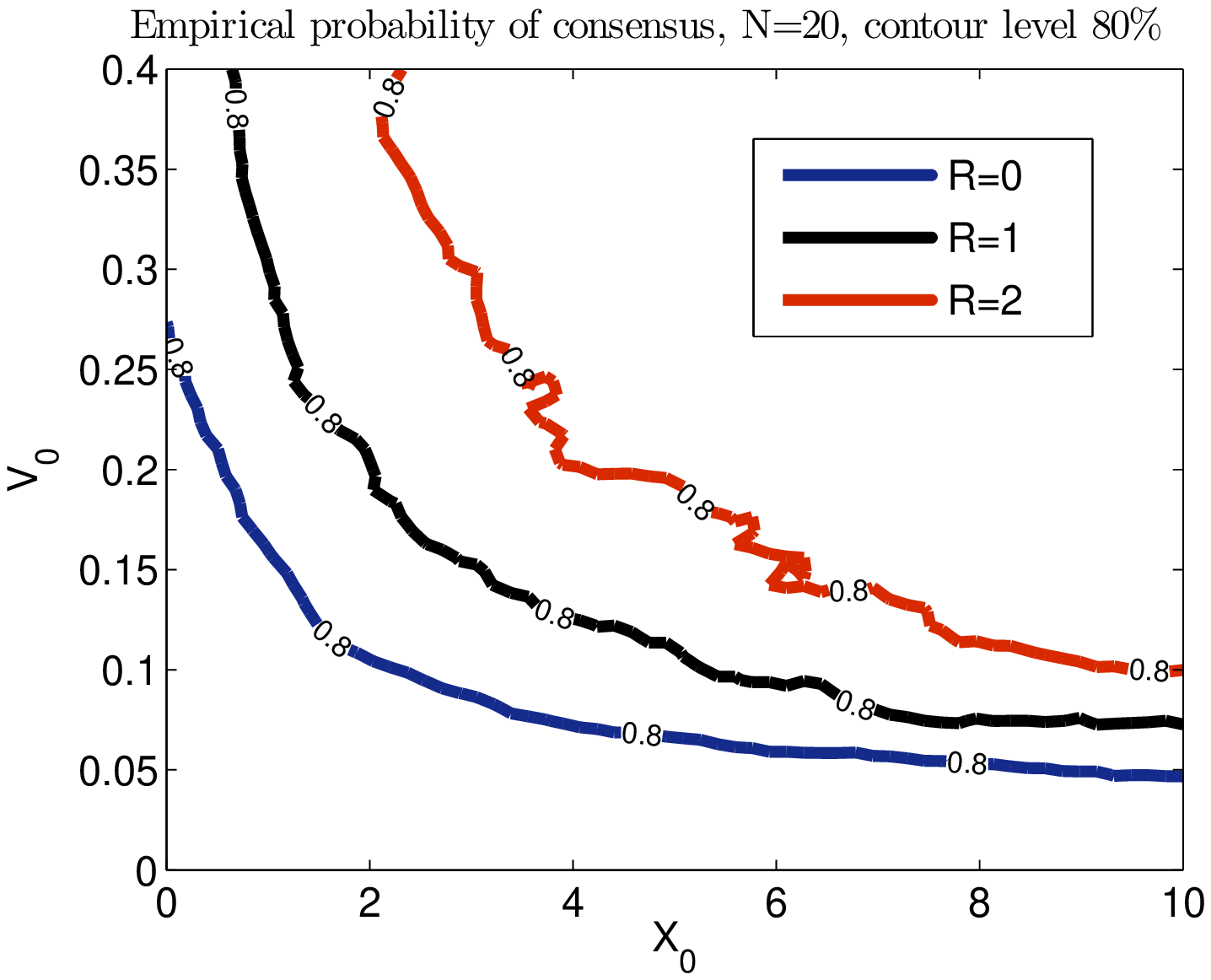,scale=0.5,clip=}
\caption{Local feedback control. Empirical contour lines for the 80$\%$ probability of consensus with different control radii.}
\label{fig:6}
\end{figure}

\section*{Concluding remarks and perspectives}
We have presented a set of feedback controllers for consensus emergence in nonlinear multi-agent systems of Cucker-Smale type. The proposed control designs address different situations concerning leader-following configurations, stabilization under perturbed information, and decentralized, local feedback control. In general, we characterize consensus emergence in every case, providing a coherent extension of the available results in the literature. Furthermore, numerical experiments assess the performance of the controllers in a consistent way. \\
Among possible future directions of research, let us mention that numerical evidence suggest that sharper consensus estimates should be possible to be derived, if the structure of the internal dynamics is more intensively used in the computations. Another natural extension of our work would be to consider the consensus emergence problem, under a decentralized control computed via an optimality-based approach.

\bibliographystyle{abbrv}
\bibliography{biblioflock}

\end{document}